\documentclass[twoside,a4paper,11pt]{article}
\usepackage[T1]{fontenc}
\usepackage[utf8x]{inputenc}
\usepackage{lmodern}
\usepackage{amsmath}
\usepackage{amsfonts}
\usepackage{mathrsfs}
\usepackage{amsthm}
\usepackage{amssymb}
\usepackage{latexsym}
\usepackage{enumerate}
\usepackage{graphicx}
\usepackage{subfigure}
\usepackage{epstopdf}
\usepackage{cite}
\usepackage{color}
\usepackage[titletoc]{appendix}
\usepackage[linkcolor=black,colorlinks,hypertexnames=false,plainpages=false,citecolor=black,filecolor=black]{hyperref}
\usepackage{geometry}

\usepackage{textcomp}
\usepackage[some]{background}
\usepackage[absolute,overlay]{textpos}
\usepackage[english]{babel}
\definecolor{color1}{RGB}{31,55, 61}
\definecolor{background}{RGB}{238,237,229}

\definecolor{comp}{RGB}{98,99,101}
\definecolor{phys}{RGB}{31,55,61}
\definecolor{math}{RGB}{133,63,72}
\definecolor{urban}{RGB}{51,95,60}

\definecolor{color1}{RGB}{133,63,72}

\usepackage{pgf,tikz}
\usetikzlibrary{arrows}

\usepackage{scrextend}

\newcommand\norm[1]{\left\lVert#1\right\rVert}
\newtheorem{prop}{Proposition}[section]
\newtheorem{definition}{Definition}[section]
\newtheorem{theorem}{Theorem}[section]
\newtheorem{lemma}{Lemma}[section]
\newtheorem{corro}{Corrolary}[section]
\theoremstyle{remark}
\numberwithin{equation}{section}
\DeclareUnicodeCharacter{1013}{-}
\begin{document}
\title{Brownian fluctuations of the interface in a system with two linear attracted components and white noises}
\author{Tran Hoa Phu}
\newcommand{\Addresses}{{
  \bigskip
  \footnotesize

  Tran Hoa Phu, \textsc{Department of Mathematics, Gran Sasso Science Institute, Viale Francesco Crispi, 7 - 67100 L'Aquila (ITALY)}\par\nopagebreak
  \textit{E-mail address}: \texttt{hoaphu.tran@gssi.it}
}}

\maketitle
\begin{abstract}
We concern the analysis of the long time behavior of interfaces in systems with two components. Each component evolves according to 1-d Allen-Cahn equation with Neumann boundary conditions, perturbed by small space-time white noise and with symmetric double well potential in the interval $[-\epsilon^{-1},\epsilon^{-1}]$. The two components interact with each other by an attractive linear force. Instantons are defined as the stationary solution of the Allen-Cahn equation without noise which connects two pure phases. We prove that for time $t=\epsilon^{-1}$, in the limit $\epsilon \rightarrow 0$, when initial states are close to an instanton, two components stay close to the same instanton, whose center moves as a Brownian motion.
\end{abstract}
\section{Introduction and main result}
 We consider the processes given as solutions of the initial value problem for stochastic partial differential system on the domain $\Omega_{\epsilon} \times [0,+\infty); \Omega_\epsilon : =[-\epsilon^{-1},\epsilon^{-1}]$ 
\begin{equation}\label{classical system}
\begin{aligned}
&\partial_t m_{1,t} =\dfrac{1}{2} \partial_{xx}m_{1,t} -V'(m_{1,t}) + \lambda(m_{2,t}-m_{1,t})+\sqrt{\epsilon} \dot{\alpha_1}(x,t), \quad x \in \Omega_{\epsilon}, t \geq 0\\
&\partial_t m_{2,t} =\dfrac{1}{2} \partial_{xx}m_{2,t} -V'(m_{2,t}) +  \lambda(m_{1,t}-m_{2,t})+ \sqrt{\epsilon} \dot{\alpha_2}(x,t), \quad x \in \Omega_{\epsilon}, t \geq 0\\
&\partial_x m_i(-\epsilon^{-1},t) = \partial_x m_i(\epsilon^{-1},t) = 0, \quad \quad t \geq 0 \quad \forall i=1,2 \\
&m_1(x,0) = m_{1,0}(x), \; m_2(x,0) = m_{2,0}(x) \quad x \in \Omega_{\epsilon}
\end{aligned}
\end{equation}
where $V'(m) = m^3 -m$ is the derivative with respect to $m$ of the double well polynomial $V(m) = \dfrac{1}{4}m^4 -\dfrac{1}{2}m^2$ and $\lambda$ is a positive parameter. In addition, $m_{1,0}$ and $m_{2,0} \in C(\Omega_\epsilon)$ are the initial datums. The terms $\dot{\alpha_1}(x,t)$ and $\dot{\alpha_2}(x,t)$ are two independent space-time white noises, i.e,  
\begin{equation}
\mathbb{E}(\dot{\alpha_i} (x,t)) = 0, \quad 
\mathbb{E}\Big(\dot{\alpha_i}(x,t)\dot{\alpha_i}(x',t')\Big)=\delta(x-x')\delta(t-t') \quad \forall i =1,2
\end{equation}

It is seen that $m_{1,t}$ and $m_{2,t}$ evole the same symmetric double well potential $V$ and they are coupled by an attractive linear force. Therefore, it is worth recalling a few result of Allen-Cahn equation on finite domain $\Omega_\epsilon$
\begin{equation}\label{Allen Cahn}
\partial_t m = \dfrac{1}{2} \partial_{xx}m - V'(m)
\end{equation} 
This PDE has two solutions $1$ and $-1$ that are two equilibrium homogeneous solution associated to two minima of the potential $V$. They are two stable solutions of the flow defined as the Cauchy problem of the equation. Furthermore, there are two special unstable stationary solutions $\pm m^{\epsilon}$ with $m^{\epsilon}$ is strictly increaing and antisymmetric, such that $m^\epsilon \rightarrow \overline{m}$ pointwise, as $\epsilon \rightarrow 0$, where $\overline{m}(x):= \tanh(x)$ is a stationary solution of the Allen-Cahn equation in the whole $\mathbb{R}$. The stability of $m^{\epsilon}$ of the flow $T_t^{\epsilon}$ defined by \eqref{Allen Cahn} was analysed by Fusco and Hale \cite{Fusco1989} and Carr and Pego \cite{Carr}. The  unstable manifolds of $m^{\epsilon}$ has dimension 1 and  the flow on it includes two orbits that connects $m^{\epsilon}$ to two stable solutions $1$ and $-1$ as $t \rightarrow \infty$. The motion along this manifold is very slow as shown in Fusco and Hale\cite{Fusco1989}. When small white noise is added to \eqref{Allen Cahn}, in the limit $\epsilon \rightarrow 0$, the analysis in \cite{BraDP}, \cite{Brassesco1994}, \cite{Brassesco1998a}, \cite{Funaki1995} shows that when initial datum is close to $m^{\epsilon}$, the solution is still close to $m^\epsilon(x + \sqrt{\epsilon}b_{T_\epsilon})$, here $b_{T_\epsilon}$ is a Brownian motion. This result for time $T_\epsilon = \epsilon^{-\gamma},\gamma<1/3$ was obtained by S. Brassesco \cite{Brassesco1994} by considering the problem in $\mathbb{R}$, in this setting, we see the infinitesimal shifts of the displacement of the instanton. The value $T_\epsilon = \epsilon^{-1}$ was derived by Presutti, Masi and Brassesco\cite{BraDP}. In addition, \cite{BraDP} has shown that the time $T_\epsilon = \epsilon^{-1}$ is needed to observe finite shifts.

Our model is an extension to the one studied in \cite{BraDP} in a sense that we extend from one component to two components and let them couple in a linear way.  In fact, we see that two components $m_{1,t}$ and $m_{2,t}$ tend to attract each other by the linear force as $\lambda > 0$. Since the crucial step used in \cite{BraDP} is  investigating the long time behavior of the linearized equation \eqref{Allen Cahn} around the instanton, it is thus expected that our system can be solved if we are are able to investigate the long time behavior of the linearized system \eqref{classical system} around the instanton. This procedure is not difficult because of linear interaction between two components. The content of this paper is to confirm this observation. We state the main theorem
\begin{theorem}\label{main theorem}
Let $m_t = \big(m_{1,t},m_{2,t}\big)$ be the solution of system \eqref{classical system} with initial condition $m_0 = \big( m_{1,0},m_{2,0}\big)$. Let $0<\zeta<1, \;\epsilon>0$ and suppose that $m_{1,0}$ and $m_{2,0}$ are continuous, satisfying  Neumann boundary condition and there is $x_{0}$ with $|x_{0}| \leq (1-\zeta)\epsilon^{-1}$
\begin{equation}
\norm{m_{i,0} - \overline{m}_{x_{0}}}_{\epsilon} \leq \sqrt{\epsilon} \quad \forall i=1,2
\end{equation}
Then there is a continuous process $\xi_t$ adapted to $m_{1,t}$ and $m_{2,t}$, such that for any $T>0$
\begin{equation}
\lim_{\epsilon \rightarrow 0^+} P^{\epsilon} \bigg( \sup_{x\in \Omega_{\epsilon}} \big|m_{i,t}(x) - \overline{m}(x-\xi_{t}) \big| > \epsilon^{1/4} \quad \forall t \leq \epsilon^{-1}T, \quad \forall  i =1,2 \bigg) = 0 
\end{equation}
Furthermore, calling $\mathbb{P}^{\epsilon}$ be the law on $C([0,T],\mathbb{R})$ of $\xi_{\epsilon^{-1} t} - x_0$. Then $\mathbb{P}^{\epsilon}$ converges weakly to a Brownian motion starting from $0$ with diffusion coefficient $3/8$. 
\end{theorem}

\section{Linearized integral system around instanton}
We refer to Walsh\cite{Walsh2006}, Faris and Jona-Lasinio \cite{Faris1982} and Da Prato and Zabczyk \cite{Zabc} for the SPDE's theory applied to \eqref{classical system}. Since the solutions are indifferentiable due to the availability of white noises so any solutions should be understood as the weak solutions. We have to rewrite the above system into integral version to make sense of white noises $\dot{\alpha_1}(x,t)$ and $\dot{\alpha_2}(x,t)$


\begin{equation}\label{eqn: int1}
\begin{aligned}
m_{1,t} &= H_t^{(\epsilon)}m_{1,0} - \int_0^t H_{t-s}^{(\epsilon)}\Big[m_{1,s}^3 - m_{1,s} + \lambda\Big(m_{1,s}-m_{2,s}\Big)\Big]ds + \sqrt{\epsilon}Z_{1,t}^{(\epsilon)} \\
m_{2,t} &= H_t^{(\epsilon)}m_{2,0} - \int_0^t H_{t-s}^{(\epsilon)}\Big[m_{2,s}^3 - m_{2,s} + \lambda\Big(m_{2,s}-m_{1,s}\Big)\Big]ds + \sqrt{\epsilon}Z_{2,t}^{(\epsilon)}
\end{aligned}
\end{equation}
where
\begin{equation}
Z_{1,t}^{(\epsilon)}(x) = \int_0^t \int_{-\epsilon^{-1}}^{\epsilon^{-1}} H_{t-s}^{(\epsilon)}(x,y)\dot{\alpha_1}(y,s)dyds
\end{equation}
\begin{equation}
Z_{2,t}^{(\epsilon)}(x) = \int_0^t \int_{-\epsilon^{-1}}^{\epsilon^{-1}} H_{t-s}^{(\epsilon)}(x,y)\dot{\alpha_2}(y,s)dyds
\end{equation}
are two independent Gaussian random variables with mean zero and variance 
\begin{equation*}
\mathbb{E}\Big(Z_{i,t}^{(\epsilon)}\Big)=0,\quad \mathbb{E}\Big({Z_{i,t}^{(\epsilon)2}(x)}\Big)  = \int_0^{t} \int_{-\epsilon^{-1}}^{\epsilon^{-1}} H_{t-s}^{(\epsilon)2}(x,y)dyds \quad \forall i=1,2
\end{equation*}
with $H^{(\epsilon)}_t$ is the Green operator for the heat equation with Neumann boundary conditions (N.b.c) in $\Omega_{\epsilon}$.

Existence and uniqueness solution of system \eqref{eqn: int1} in bounded domain follows by extending results of Faris and Jona-Lasinio \cite{Faris1982}. We say that $(m_{1,t},m_{2,t})$ is a solution of \eqref{classical system} if it satisfies \eqref{eqn: int1} and it defines a continuous process $m_{i,t},\; t\geq 0,$ with values in $C(\Omega_\epsilon)$.
\begin{definition}
It is convenient to extend the process $\big(m_{1,t},m_{2,t}\big)$ of the system \eqref{eqn: int1} to the whole real line since we could exploit stability instanton properties. The useful technique used for extension was introduced in details by ref. $[1]$. We now recall it,  
for any continuous function $m$ in $ \Omega_{\epsilon}$, we define its extension $\widehat{m}$ to $\mathbb{R}$ by reflecting $m$ through $\epsilon^{-1}$ and then extending to $\mathbb{R}$ with period $4\epsilon^{-1}$.
\begin{equation}\label{extended function}
\begin{aligned}
\widehat{m}(x) := \sum_{k \in \mathbb{Z}} \bigg( &m(x-4k\epsilon^{-1})\chi_{[(4k-1)\epsilon^{-1},(4k+1)\epsilon^{-1}]}(x)\\
& + m((4k+2)\epsilon^{-1}-x)\chi_{[(4k+1)\epsilon^{-1},(4k+3)\epsilon^{-1}]}(x)\bigg)
\end{aligned}
\end{equation}
We say that a function $\widehat{m} \in C^0(\mathbb{R})$ defined on the whole line satisfies N.b.c if it is the extension of a function in $\Omega_{\epsilon}$. 

We define
\begin{equation*}
Z_{1,t} = \widehat{Z}_{1,t}^{(\epsilon)}, \quad Z_{2,t} = \widehat{Z}_{2,t}^{(\epsilon)} 
\end{equation*}
\end{definition}
Let $H_t$ be the Green operator for the heat equation on the whole line.
\begin{equation}\label{eqn 2.6}
H_t^{(\epsilon)}m(x) = \int_{-\infty}^{\infty}\dfrac{1}{\sqrt{2 \pi t}}\exp(-(x-y)^2/2t) \widehat{m}(y)dy \equiv H_t \widehat{m}(x)
\end{equation}
We now deal with extended process $(m_{1,t},m_{2,t})$ on unbounded domain $\mathbb{R}\times \mathbb{R^+}$
\begin{prop}\label{system for m-i,t for whole line}

For any $\epsilon >0$, for any $m_{1,0},m_{2,0} \in C^0(\mathbb{R})$ that satisfy N.b.c in $
\Omega_{\epsilon}$ and for any $Z_{1,t}(x),Z_{2,t}(x)$ continuous in both variables and satisfying N.b.c, there is a unique continuous solution $m_{1,t}$ and $m_{2,t}$ 
\begin{equation}\label{eqn: integral_on_R_1}
\begin{aligned}
&m_{1,t} = H_t m_{1,0} - \int_0^t H_{t-s}\Big[m_{1,s}^3 - m_{1,s} + \lambda\Big(m_{1,s}-m_{2,s}\Big)\Big]ds + \sqrt{\epsilon}Z_{1,t}\\
&m_{2,t} = H_tm_{2,0} - \int_0^t H_{t-s}\Big[m_{2,s}^3 - m_{2,s} + \lambda\Big(m_{2,s}-m_{1,s}\Big)\Big]ds + \sqrt{\epsilon}Z_{2,t}
\end{aligned}
\end{equation}
Moreover $m_{1,t} = \widehat{m}_{1,t}^{(\epsilon)},  m_{2,t} = \widehat{m}_{2,t}^{(\epsilon)}$ where $m_{1,t}^{(\epsilon)}$ and $m_{2,t}^{(\epsilon)}$ solves \eqref{eqn: int1} with $Z_{1,t}^{(\epsilon)}$,   $Z_{2,t}^{(\epsilon)}$ obtained by restricting $Z_{1,t},Z_{2,t}$ to $\Omega_{\epsilon}.$ In addition, $m_{1,0}^{(\epsilon)}$ and  $m_{2,0}^{(\epsilon)}$ also obtained by restricting $m_{1,0}$ and $m_{2,0}$ to $\Omega_{\epsilon}$.
\end{prop}
\begin{proof}
By extending the results in \cite{Doering1987}, it is not difficult to show that \eqref{eqn: integral_on_R_1} has a uniqueness continuous solution, which also satisfies N.b.c in $\mathbb{R}$ since $m_{1,0},m_{2,0}$ and $Z_{1,t},Z_{2,t}$ satisfy N.b.c. By \eqref{eqn 2.6}, the restriction to $[-\epsilon^{-1},\epsilon^{-1}]$ solves \eqref{eqn: int1}, whose solution is unique.  
\end{proof}
In order to investigate the stability of process around instantons, it is essential to linearize the system around instanton. The following Proposition can be similarly proved in Prop 2.5, \cite{BraDP}.  

\begin{prop}\label{prop : 2D linearized sys} 

For any given $\epsilon > 0$, $m_{1,0},m_{2,0}$ and $Z_{1,t},Z_{2,t}$ as in Proposition \ref{system for m-i,t for whole line}. Then, $m_{1,t}$ and $m_{2,t}$solve \eqref{eqn: integral_on_R_1} if and only if $u_t:= m_{1,t}-\overline{m}_{x_0},\;v_t:= m_{2,t} -\overline{m}_{x_0}$ are solution of the following system
\begin{equation}\label{eqn: uu}
\begin{aligned}
&u_t = g_{t,x_{0}}u_0-\int_0^t  g_{t-s,x_{0}}\Big( 3\overline{m}_{x_0}u_s^2 + u_s^3 - \lambda(v_s-u_s) \Big)ds + \sqrt{\epsilon}W_{1,t,x_{0}}\\
&v_t = g_{t,x_{0}}v_0-\int_0^t  g_{t-s,x_{0}} \Big( 3\overline{m}_{x_0}v_s^2 + v_s^3 - \lambda(u_s - v_s) \Big)ds + \sqrt{\epsilon}W_{2,t,x_{0}}
\end{aligned}
\end{equation}
where $g_{t,x_0}$ is the operator that solves the initial value problem for 
\begin{equation}
\partial_t u = L_{x_0}u:= \dfrac{1}{2}\partial_{xx}u - V^{''}(\overline{m}_{x_0})u
\end{equation}
and Gaussian noises $W_{i,t,x_0}$ have the form

\begin{equation}\label{noise W_t}
\begin{aligned}
W_{i,t,x_0}(x)= \int_0^t \int_{-\epsilon^{-1}}^{\epsilon^{-1}} \sum_{k\in \mathbb{Z}}\Big( &g_{t-s,x_0}(x,y+4k\epsilon^{-1})\\
&+g_{t-s,x_0}(x,4k\epsilon^{-1}+2\epsilon^{-1}-y) \Big)\dot{\alpha_i}(y,s)dyds
\end{aligned}
\end{equation}
\end{prop}

The "spectral gap" of the operator $g_{t,x_0} = e^{L_{x_0}t}$ was mentioned in \cite{BraDP}. $L_{x_0}$ is a self-adjoint operator in $L^2(\mathbb{R},dx)$, as a result, its spectrums are real. Moreover, $L_{x_0}$ has an eigenvalue 0 associated to eigenvector $\overline{m}_{x_0}'$ and all of other spectrums lie on negative axis strictly away from 0. As a result, the solution of the equation $\partial_t u = L_{x_0}u$ mainly evolves in the direction of $\overline{m}_{x_0}'$ in $L^2(\mathbb{R},dx)$. Moreover, this result can be extended to $L^{\infty}(\mathbb{R},dx)$ in \cite{BraDP}. Formally, we refer again to Theorem 2.4, \cite{BraDP}:
\begin{prop}\label{1 asymptotic of operator g}.
There are $\alpha > 0$ and $c^*$ such that for any $\phi \in C^0(\mathbb{R})$ and $x_0 \in \mathbb{R}$
\begin{equation}
\norm{g_{t,x_0}(\phi - \langle \phi,\tilde{m}_{x_0}'\rangle\tilde{m}_{x_0}')}_{\infty} \leq c^*e^{-\alpha t}\norm{\phi - \langle \phi,\tilde{m}_{x_0}'\rangle\tilde{m}_{x_0}'}_{\infty}
\end{equation}
Here $\langle.,.\rangle$ is  inner product in $L^2(\mathbb{R})$ and normalized eigenfunction $\tilde{m}_{x_0}' := \dfrac{\sqrt{3}}{2}\overline{m}_{x_0}'$.
\end{prop}
Above useful property helps us to derive the long time behavior of linearized system \eqref{eqn: uu} without noises by diagonalizing the system $(u_{1,t},u_{2,t})$ to $\bigg(\dfrac{u_{1,t}+u_{2,t}}{2},\dfrac{u_{1,t}-u_{2,t}}{2} \bigg)$. 
\begin{corro}\label{corro 2D}
Let $u_1,u_2$ be solution of the following system on $\mathbb{R} \times [0,+\infty)$
\begin{equation}
\begin{aligned}
&\partial_t u_1 = 1/2\; \partial_{xx}u_1 - V''(\overline{m}_{x_0})u_1 + \lambda (u_2 - u_1)\\
&\partial_t u_2 = 1/2 \; \partial_{xx}u_2 - V''(\overline{m}_{x_0})u_2 + \lambda (u_1 - u_2)\\
&u_1(x,0)= u_{1,0}(x),\quad u_2(x,0) = u_{2,0}(x)
\end{aligned}
\end{equation}
Then both solutions mainly evolve to the same direction $\tilde{m}_{x_0}'$ exponentially in the sup norm, i.e, $\forall i=1,2 $ 
\begin{equation}\label{same convergence}
\quad \lim_{t \rightarrow \infty}\norm{u_{i,t} - \bigg \langle \dfrac{u_{1,0}+u_{2,0}}{2},\tilde{m}_{x_0}' \bigg \rangle\tilde{m}_{x_0}'}_\infty = 0
\end{equation}
\end{corro}
\begin{proof}
Denoting by $U_t^+ = u_{1,t} + u_{2,t},\quad U_t^- = u_{1,t} - u_{2,t}$, it is easy to derive that they evolve as the following system
\begin{equation}
\begin{aligned}
&\partial_t U_t^+ = \dfrac{1}{2} \partial_{xx}U_t^+ - V''(\overline{m}_{x_0})U_t^+ \\
&\partial_t U_t^- = \dfrac{1}{2} \partial_{xx}U_t^- - V''(\overline{m}_{x_0})U_t^- - 2\lambda U_t^- 
\end{aligned}
\end{equation}
We obtain, by Theorem \ref{1 asymptotic of operator g}, there are $c,\alpha >0$ such that 
\begin{equation}
\norm{U_t^+ - \langle U_0^+,\tilde{m}_{x_0}'\rangle \tilde{m}_{x_0}'}_\infty \leq c e^{-\alpha t}\norm{U_0^+ - \langle U_0^+ ,\tilde{m}_{x_0}' \rangle  }_\infty 
\end{equation}
In a similar way, by changing variable, we could prove that
\begin{equation}
\norm{U_t^- - e^{-2\lambda t}\langle U_0^- , \tilde{m}_{x_0}'\rangle \tilde{m}_{x_0}'}_\infty \leq c e^{-\alpha t - 2\lambda t} \norm{U_0^- - \langle U_0^- , \tilde{m}_{x_0}'\rangle \tilde{m}_{x_0}'}_\infty
\end{equation}
$u_{1,t}$ and $u_{2,t}$ could be estimated through $U_t^+$ and $U_t^-$ as follows
\begin{equation}
\begin{aligned}
&\norm{u_{1,t} - \langle \dfrac{U_0^+}{2},\tilde{m}_{x_0}' \rangle \tilde{m}_{x_0}' - e^{-2\lambda t} \bigg \langle \dfrac{U_0^-}{2},\tilde{m}_{x_0}' \bigg \rangle \tilde{m}_{x_0}'}_\infty \\
&\leq  \dfrac{1}{2}\norm{U_t^+ - \langle U_0^+ , \tilde{m}_{x_0}'\rangle \tilde{m}_{x_0}' }_\infty + \dfrac{1}{2} \norm{U_t^- - e^{-2\lambda t}\langle U_0^-,\tilde{m}_{x_0}' \rangle \tilde{m}_{x_0}'}_\infty \\
&\leq \dfrac{1}{2}c e^{-\alpha t}\norm{U_0^+ - \langle U_0^+ ,\tilde{m}_{x_0}' \rangle \tilde{m}_{x_0}'  }_\infty + \dfrac{1}{2}c e^{-\alpha t - 2\lambda t} \norm{U_0^- - \langle U_0^- , \tilde{m}_{x_0}'\rangle \tilde{m}_{x_0}'}_\infty 
\end{aligned}
\end{equation}
Thus,
\begin{equation}
\begin{aligned}
&\norm{u_{1,t} -  \bigg \langle \dfrac{U_0^+}{2},\tilde{m}_{x_0}' \bigg \rangle \tilde{m}_{x_0}'}_\infty \leq e^{-2\lambda t} \norm{\bigg \langle \dfrac{U_0^-}{2},\tilde{m}_{x_0}' \bigg \rangle \tilde{m}_{x_0}'}_\infty \\
 &+ \dfrac{1}{2}c e^{-\alpha t}\norm{U_0^+ - \langle U_0^+ ,\tilde{m}_{x_0}' \rangle \tilde{m}_{x_0}'  }_\infty + \dfrac{1}{2}c e^{-\alpha t - 2\lambda t} \norm{U_0^- - \langle U_0^- , \tilde{m}_{x_0}'\rangle \tilde{m}_{x_0}'}_\infty 
\end{aligned}
\end{equation}
and this implies \eqref{same convergence}.
\end{proof}
The kernel $g_{t,x_0}(x,y)$ of the operator $g_{t,x_0}$ was studied in  Lemma A.4,\cite{Brassesco1998} that it is positive and uniformly integrable respect to $y$:
\begin{lemma}
Let $x_0 \in \mathbb{R}$ and let $g_{t,x_0}(x,y)$ be the fundamental solution of the equation $\partial_t u =\dfrac{1}{2}\partial_{xx}u - V''(\overline{m}_{x_0})u$. Then there exists a positive constant $\kappa$ such that:
\begin{equation}\label{asymptotic of operator g}
\begin{aligned}
&g_{t,x_0}(x,y) \geq 0 \quad \text{for any} \quad x,y \in \mathbb{R} \\
& \sup_{x \in \mathbb{R}} \sup_{t>0} \int_\mathbb{R} g_{t,x_0}(x,y)dy \leq \kappa
\end{aligned}
\end{equation}
\end{lemma}
The following proposition stated in Prop 5.4,\cite{BraDP} provides a representation of the noises $W_{i,t,x_0}$ in the direction $\overline{m}_{x_0}'$ and estimate their boundedness unitl time $t \leq \epsilon^{-2}$.
\begin{prop}\label{prop noise}
The processes $W_{i,t,x_0}(i=1,2)$ is represented as below
\begin{equation}
W_{i,t,x_0} =: B_{i,t,x_0}\tilde{m}_{x_0}' + R_{i,t} \quad (D:=3/4,\; \tilde{m}_{x_0}':=\sqrt{D}\overline{m}_{x_0}')
\end{equation}
with the following properties.
$B_{i,t,x_0}=\langle W_{i,t,x_0},\tilde{m}_{x_0}'\rangle$ is a process adapted to $W_{i,t,x_0}$, its laws is the law of a Brownian motion with diffusion coefficient
\begin{equation}
D_{\epsilon} = \int_{-\epsilon^{-1}}^{\epsilon^{-1}}\Big[ \sum_{k \in \mathbb{Z}}(\tilde{m}_{x_0}'(y+4k\epsilon^{-1})+\tilde{m}_{x_0}'(4k\epsilon^{-1}+2\epsilon^{-1}-y) \Big]^2dy
\end{equation}
and there is a constant c such that
\begin{equation}
\lvert D_{\epsilon}-1\rvert \leq c e^{-\epsilon^{-1}}
\end{equation}
For any $a>0$ let
\begin{equation}
G_{i,\epsilon}(a) := \lbrace \norm{W_{i,t,x_0}}_{\infty} \leq \epsilon^{-a}(t \vee 1)^{1/2}, \norm{R_{i,t}}_{\infty} \leq \epsilon^{-a}, \text{for all}\; t \leq \epsilon^{-2} \rbrace
\end{equation}
Then for any $n \geq 1$ there is $c_n$ so that
\begin{equation}\label{2 bound noise}
P^{\epsilon}(G_{i,\epsilon}(a)) \geq 1 - c_n \epsilon^n
\end{equation}
\end{prop}
\section{Sketch of the proof}
Before giving heuristic proof of Theorem \ref{main theorem}, it is necessary to recall stability properties of instantons $M = \lbrace \overline{m}_\xi , \xi \in \mathbb{R} \rbrace$ in deterministic case. Let us write Allen-Cahn (AC) equation in unbounded domain 
\begin{equation}\label{AC 1}
\dfrac{\partial m(x,t)}{\partial t} = \Delta m(x,t) - V'(m(x,t)) \quad (x,t) \in \mathbb{R} \times \mathbb{R}
\end{equation}
The analysis shows that, under AC evolution, $M$ attracts exponentially fast, in sup norm, all functions which are in small neighborhood of $M$, i.e, if $m$ is close enough to $\overline{m}_\xi$ in sup norm, then there is $\xi'\in \mathbb{R}$ such that for all $t$, $\norm{m(t) - \overline{m}_{\xi'}}_\infty \approx e^{-at}$ for some $a>0$. This result is no longer true if small noises $\epsilon^{1/2} \dot{\alpha}(x,t)$ is added in \eqref{AC 1} as shown in \cite{BraDP}. Indeed, we have
\begin{equation}
\sup_{t\leq \epsilon^{-1}\tau} \text{dist} (m(t),M) \leq \epsilon^{1/2-a}
\end{equation}
with "large probability" for $\epsilon$ small. One important technique to study instantons's stability of \eqref{AC 1} when small noises $\epsilon^{1/2}Z_t$ is added is the notion of center. $\xi$ is called a center of $m$ if
\begin{equation}\label{1 center}
\langle m - \overline{m}_\xi, \overline{m}_{\xi}'\rangle = 0
\end{equation}
The role of "center" is interpreted as follow: if $ m(x)$ has a center $\xi$, then the linearized evolution of \eqref{AC 1} around $\overline{m}_\xi$ is
\begin{equation}
 \partial_t \phi = L_\xi \phi = \dfrac{1}{2} \Delta \phi - V''(\overline{m}_\xi)\phi
\end{equation} 
By Proposition \ref{1 asymptotic of operator g}, $\phi = m - \overline{m}_\xi$ mainly evolve along $\langle \phi, \overline{m}_\xi' \rangle \overline{m}_\xi' = \langle m-\overline{m}_\xi, \overline{m}_\xi'\rangle \overline{m}_\xi'$, this term, however, vanishes by \eqref{1 center}. Hence, the center of $m$ is also the center of the instanton to which the linearized AC evolution converges. Therefore, it is natural to decompose $m_t - \overline{m}_{\xi(t)}$ ($\xi(t)$ is center of $m_t$) into the basis along $\overline{m}_{\xi(t)}'$ since all other directions are exponentially damped. Since the noises $\epsilon^{1/2}Z_t$ dominates at the beginning of evolution of $m_t$, so non linear terms may be neglected for short time $t_\epsilon = \epsilon^{-1/10}$ 
\begin{equation}
dm_{t_\epsilon} \approx \dfrac{3}{4}\epsilon^{1/2} \overline{m}_{\xi(t_\epsilon)}' \langle \overline{m}_{\xi(t_\epsilon)}',dZ_{t_\epsilon} \rangle
\end{equation}
 Furthermore, from $m_{t_\epsilon} \approx \overline{m}_{\xi(t_\epsilon)}$ we get
 \begin{equation}
 dm_{t_\epsilon} \approx \overline{m}_{\xi(t_\epsilon)}'d\xi(t_\epsilon)
 \end{equation}
Thus,
\begin{equation}
d\xi(t_\epsilon) \approx \dfrac{3}{4} \epsilon^{1/2} \langle \overline{m}_{\xi(t_\epsilon)}',dZ_{t_\epsilon} \rangle
\end{equation}
Moreover, $\langle \overline{m}_{\xi(t_\epsilon)}',dZ_{t_\epsilon} \rangle$ is a Brownian motion $B_{t_\epsilon}$ as shown in Prop 5.4,\cite{BraDP}. Finally,
\begin{equation}\label{xi Brownian}
d\xi_{t_\epsilon} \approx \dfrac{3}{4}\epsilon^{1/2}B_{t_\epsilon}
\end{equation}
Since the magnititue of Brownian motion $B_t \approx \sqrt{t}$ so $\xi_{t_\epsilon} \approx \epsilon^{1/2-1/20}$ is infinitesimal, and \eqref{xi Brownian} is expected to hold much longer time $T_\epsilon^* = \epsilon^{-1}$ since $\xi_{T_\epsilon^*} \approx \epsilon^{1/2}\sqrt{T_\epsilon} \approx O(1)$.  By controlling the bounded of noises through considering spatial domain in $[-\epsilon^{-1},\epsilon^{-1}]$,  above process can be iterated up to time $T= \epsilon^{-1}t$, $\xi(\epsilon^{-1}t)$ is a Brownian motion with diffusion coefficient $3/4$ as shown in \cite{BraDP}.

The same idea can also be applied to our system \eqref{eqn: integral_on_R_1} with slightly different. The linearized system for $u_{1,t} = m_{1,t} - \overline{m}_{x_0} $ and $u_{2,t} =m_{2,t} - \overline{m}_{x_0}$ is
\begin{equation}
\begin{aligned}
\partial_t u_{1,t} &= \dfrac{1}{2}\partial_{xx} u_{1,t} - V''(\overline{m}_{x_0})u_{1,t} + \lambda (u_{2,t}-u_{1,t})\\
\partial_t u_{2,t} &= \dfrac{1}{2}\partial_{xx} u_{2,}t - V''(\overline{m}_{x_0})u_{2,t} + \lambda (u_{1,t}-u_{2,t})
\end{aligned}
\end{equation}
By Corollary \ref{corro 2D}, we get for every $x \in \mathbb{R}$
\begin{equation}
u_{i,t}(x) \rightarrow \dfrac{3}{4} \bigg\langle \dfrac{u_{1,0}+u_{2,0}}{2}, \overline{m}_{x_0}' \bigg\rangle \overline{m}_{x_0}'(x) = \dfrac{3}{4} \bigg\langle \dfrac{m_{1,0}+m_{2,0}}{2}- \overline{m}_{x_0}, \overline{m}_{x_0}'  \bigg \rangle \overline{m}_{x_0}'(x)
\end{equation}
In order to get $u_{i,t}$ vanish, it is then natural to define $x_0$ such that
\begin{equation}
\bigg \langle \dfrac{m_{1,0}+m_{2,0}}{2} - \overline{m}_{x_0}, \overline{m}_{x_0}' \bigg\rangle = 0
\end{equation}
and $x_0$ is call to be the center of $\dfrac{m_{1,0}+m_{2,0}}{2}$. In short, the center of $\dfrac{m_{1,0}+m_{2,0}}{2}$ is also the center of the instanton to which the linearized evolution converge. As a result, we deal with the centers $\xi(t)$ of avarage $m_{1,t}$ and $m_{2,t}$, i.e, $\dfrac{m_{1,t}+m_{2,t}}{2}$ rather than centers of individual components. Using similar argument as above, we have, for short time $t_\epsilon  = \epsilon^{-1/10}$
\begin{equation}
dm_{i,t_\epsilon} \approx \dfrac{3}{4}\epsilon^{1/2} \overline{m}_{\xi(t_\epsilon)}' \Bigg\langle \overline{m}_{\xi(t_\epsilon)}',d\bigg( \dfrac{Z_{1,t_\epsilon} + Z_{2,t_\epsilon)}}{2} \bigg) \Bigg\rangle
\end{equation}
 Furthermore, from $m_{i,t_\epsilon} \approx \overline{m}_{\xi(t_\epsilon)}$ we get
 \begin{equation}
 dm_{i,t_\epsilon} \approx \overline{m}_{\xi(t_\epsilon)}'d\xi(t_\epsilon)
 \end{equation}
Thus,
\begin{equation}
d\xi(t_\epsilon) \approx \dfrac{3}{4}\epsilon^{1/2} \overline{m}_{\xi(t_\epsilon)}' \Bigg\langle \overline{m}_{\xi(t_\epsilon)}',d\bigg( \dfrac{Z_{1,t_\epsilon} + Z_{2,t_\epsilon)}}{2} \bigg) \Bigg\rangle
\end{equation}
and by iterating this procedure, $\xi(\epsilon^{-1}t)$ behaves as a Brownian motion with diffusion coefficient $3/8$.
\section{Fluctuations of the instantons}
The notion "center" introduced in \cite{BraDP} plays an important role in the proof. We now recall it and then explain its meaning.
\begin{definition}
The function $m \in C^0(\mathbb{R})$ has a center $x_0$ if 
\begin{equation}
\langle m-\overline{m}_{x_0}, \overline{m}_{x_0}' \rangle_{L^2} = 0
\end{equation}
\end{definition}
The role of "center" is interpreted as follow:  if  $\dfrac{m_{1,0}+m_{2,0}}{2}$ has a center $x_0$, then the linearized evolution around $\overline{m}_{x_0}$ states that $m_{1,t}-\overline{m}_{x_0}$ and $m_{2,t}-\overline{m}_{x_0}$ mainly evolve in the direction $\bigg\langle \dfrac{m_{1,0}+m_{2,0}}{2},\tilde{m}_{x_0}' \bigg\rangle \tilde{m}_{x_0}'$. As a result, they decay exponentially fast since $x_0$ is a center of $\dfrac{m_{1,0}+m_{2,0}}{2}$. Hence the center of $\dfrac{m_{1,0}+m_{2,0}}{2}$ is also the center of the instanton to which the linearized evolution converge.

We recall basic properties of center of a function. When a function is close enough to some instanton, its center is well-defined as in Prop 3.2,\cite{BraDP}:
\begin{prop} \label{propositions for center}
There are $\delta > 0$ and, given any $\zeta'$ and $\zeta$ such that $0 < \zeta' < \zeta <1$, there are $c $ and $\epsilon_0$ so that for any $0<\epsilon \leq \epsilon_0$ and any $|x_0| \leq (1-\zeta)\epsilon^{-1}$ the following holds. Let $m \in C^0(\mathbb{R}), \; \norm{m}_{\infty} \leq 2$, and\
\begin{equation}
\norm{m-\overline{m}_{x_0}}_{\epsilon} := \sup_{x \in [-\epsilon^{-1},\epsilon^{-1}]} \lvert m(x)-\overline{m}_{x_0}(x) \rvert \leq \delta
\end{equation}
Then

(1) $m$ has a center $\xi$ in $\Omega_{\epsilon}$,
\begin{equation}\label{first property of center}
|x_0 - \xi| \leq c(\norm{m - \overline{m}_{x_0}}_{\epsilon} + e^{-\epsilon^{-1}\zeta'})
\end{equation}
and $\xi$ is unique in $\lbrace |x| \leq (1-\zeta')\epsilon^{-1} \rbrace$

(2) 
\begin{equation} \label{second property  of center}
\bigg\lvert \xi - \bigg(x_0 - 3/4\langle \overline{m}_{x_0}',m-\overline{m}_{x_0} \rangle \bigg)  \bigg\rvert \leq c(\norm{m-\overline{m}_{x_0}}_\epsilon^2 + e^{-\zeta' \epsilon^{-1}})
\end{equation}

(3) Let $m^* \in C^0(\mathbb{R}), \; \norm{m^*}_{\infty} \leq 2$ and
\begin{equation}
\norm{m^*-m}_{\epsilon} \leq \delta
\end{equation} 
Then $m^*$ has a unique center $\xi^*$ in $\lbrace |x| \leq (1-\zeta')\epsilon^{-1} \rbrace$ and
\begin{equation}\label{3.7}
|\xi - \xi^*| \leq c \int \overline{m}_{x_0}'(x)|m^*(x) - m(x)|dx
\end{equation}

(4) If $m$ satisfies N.b.c in $\Omega_{\epsilon}$, then $m$ has unique center in $\Omega_{\epsilon}$.
\end{prop}
\begin{definition}
Given $m \in { C(\mathbb{R})}$, we define $\xi(m)$ as the center of 
$m$ in case $m$ satisfies the conditions of Proposition \ref{propositions for center} with some $\zeta >0$ and we say that $\xi(m)$ is proper. Otherwise we set $\xi(m) = 0$.  The estimate \eqref{same convergence} suggests us to define
\begin{equation}
\quad  \xi_{t} := \xi \bigg(\dfrac{m_{1,t}+m_{2,t}}{2}\bigg)
\end{equation}
\end{definition}
We now present the procedure to conclude the theorem.

Firstly, we shall see that when initial datum is close to the same instanton, then it will get "much closer" to an instanton up to time $t=\epsilon^{-b}$ ($b<1/10$)
\begin{prop}\label{prop 3.4}
For any $0 < \zeta < 1$ and $0 < a \leq 1/4$ there are positive constants $C$ and $b<1/10,$ and given $n,c_n$, so that the following holds. Suppose that for $m_{1,0},m_{2,0}$ satisfy N.b.c in $\Omega_{\epsilon}$, there is $x_0: |x_0| \leq (1-\zeta)\epsilon^{-1}$ and
\begin{equation}\label{m - instanton}
\norm{m_{i,0}-\overline{m}_{x_0}}_\epsilon \leq \epsilon^{1/4} \quad \forall i=1,2
\end{equation} Denote by $(m_{1,t},m_{2,t})$ the process starting from $(m_{1,0},m_{2,0})$ and $\xi = \xi_{\epsilon^{-b}}$, then
\begin{equation}
P^{\epsilon} \bigg(\sup_{t\leq \epsilon^{-b}} \norm{m_{i,t}-\overline{m}_{x_{0}}}_{\epsilon} \leq C{\epsilon^{1/4}}; \norm{m_{i,\epsilon^{-b}} - \overline{m}_{\xi}}_{\epsilon}\leq \epsilon^{1/2-a} \; \forall i=1,2\bigg) \geq 1 - c_n\epsilon^n 
\end{equation}
\end{prop}
By iterating above procedure, the solution will get close to instantons until time $t=\epsilon^{-1}$.
\begin{prop}\label{Proposistion estimate center for longer time}
Let $\zeta, a, b$ and $m_0$ be as in Proposition \ref{prop 3.4}. Then there is $c'$ and given $n, c_n$ so that, setting $s_k = k\epsilon^{-b}, \; k \in \mathbb{N},$
\begin{equation}\label{eqn 3.13}
P^{\epsilon} \bigg( \sup_{\epsilon^{-b} \leq s_k \leq \epsilon^{-1}}\norm{m_{i,s_k}-\overline{m}_{\xi_{s_k}}}_{\epsilon} \leq \epsilon^{1/2-a} \quad \forall i=1,2 \bigg) \geq 1 -c_n \epsilon^n
\end{equation}
\begin{equation}\label{eqn 3.14}
P^{\epsilon} \bigg( |\xi_t - x_0| \leq c'(1 \vee t)\epsilon^{1/4} \quad \forall t \leq \epsilon^{-1} \bigg) \geq 1 - c_n\epsilon^n
\end{equation}
\end{prop}
The last step is to show that the center $\xi_{\epsilon^{-1}t} - x_0$ behave as Brownian motion in the limit $\epsilon \rightarrow 0$
\begin{prop} \label{weak convergence}
Given any $\zeta > 0$ and $\epsilon > 0$, let $(m_{1,t},m_{2,t})$ be the process that starts from $(m_{1,0},m_{2,0})$, with $|x_0| \leq (1-\zeta)\epsilon^{-1}$. Define
\begin{equation}
X_{t}^{\epsilon} := \xi_{\epsilon^{-1}t} - x_0
\end{equation}
and let $P^{\epsilon}$ be the law on $C(\mathbb{R}_+,\mathbb{R})$ of $X_{t}^{\epsilon}$. Then $P^{\epsilon}$ converges weakly as $\epsilon \rightarrow 0$ to $P$ the law of the Brownian motion with diffusion coefficient $D=3/8$ that start from 0.
\end{prop}

Finally, Theorem \ref{main theorem} is followed from Proposition \ref{prop 3.4}, \ref{Proposistion estimate center for longer time} and \ref{weak convergence}.  

\section{Proofs of part 4}
This section gives formal proofs of Proposition \ref{prop 3.4} and \ref{weak convergence}. We omit the proof of Proposition \ref{Proposistion estimate center for longer time} since it is similar to Lemma 3.5, \cite{BraDP}.   
\subsection{Proof of Proposition \ref{prop 3.4}}
\begin{proof}
 From \eqref{m - instanton}, we get
\begin{equation}
 \norm{\dfrac{m_{1,0}+m_{2,0}}{2}- \overline{m}_{x_0}}_\epsilon \leq \epsilon^{1/4}
\end{equation}
Thus, by \eqref{first property of center}
\begin{equation}
|\xi_{0} - x_0| \leq c\epsilon^{1/4}
\end{equation}

for some $c>0$. As a result, $\forall i=1,2$ we have
\begin{equation}\label{1}
\norm{m_{i,0} - \overline{m}_{\xi_{0}}}_{\epsilon} \leq \norm{m_{i,0}-\overline{m}_{x_0}}_\epsilon + \norm{\overline{m}_{x_0} - \overline{m}_{\xi_0}}_\epsilon \leq C\epsilon^{1/4}, \quad C:= 1 + c
\end{equation}
Due to the fact that when $\norm{m-\overline{m}}_\epsilon$ is small and $m$ satisfies N.b.c then $\norm{m-\overline{m}}_\infty$ is not small. This fact can be passed by using "Barrier lemma" for "localization". We define for $i=1,2$
\begin{equation}
\hat{m}_{i,0}(x) = \left\{\begin{array}{l}
m_{i,0}(x) \quad \text{if} \quad |x-\xi_{0}| \leq 10^{-4}\zeta \epsilon^{-1} \\
1 \quad \quad \text{if} \quad x-\xi_0 > 10^{-4}\zeta \epsilon^{-1} +1  \\
-1 \quad \text{if} \quad x- \xi_0 < -10^{-4} \zeta \epsilon^{-1} - 1   
\end{array}\right.
\end{equation}
The remaining part of $\hat{m}_{i,0}(x)$ can be completed by linear interpolation. Let $\hat{m}_{t}$ and $m_{t}$ be the solutions of \eqref{eqn: integral_on_R_1} with inital data respectively $\hat{m}_{0}$ and $m_{0}$ (and same noise). By Proposition \ref{proposition of barrier lemma}, for any $n$ there is $c_n$ so that 
\begin{equation} \label{4.17}
P^{\epsilon} \bigg(\sup_{t \leq \epsilon^{-b}} \sup_{|x-\xi_{0}|\leq 10^{-5}\zeta \epsilon^{-1}} |m_{i,t}(x) - \hat{m}_{i,t}(x)| \leq c_n \epsilon^n \quad \forall i=1,2 \bigg) \geq 1 - c_n\epsilon^n
\end{equation} 
We will prove that there is $\hat{a} < a$ and for any $n, c_n$ so that
\begin{equation}\label{4.18}
\begin{aligned}
P^{\epsilon}\bigg( \sup_{t \leq \epsilon^{-b}} \sup_{|x-\xi_{0}| \leq 10^{-5}\zeta \epsilon^{-1}} |\hat{m}_{i,t}(x) - \overline{m}_{\hat{x}_{0}}(x)| < C\epsilon^{1/4}; \\
\sup_{|x-\xi_{0}| \leq 10^{-5}\zeta \epsilon^{-1}}|\hat{m}_{i,\epsilon^{-b}}(x)-\overline{m}_{\hat{x}_{0}}(x)| < \epsilon^{1/2-\hat{a}} \quad \forall i=1,2\bigg) \geq 1 -c_n \epsilon^n
\end{aligned}
\end{equation}
where $\hat{x}_0 := \xi(\dfrac{\hat{m}_{1,0}+\hat{m}_{2,0}}{2})$, by \eqref{3.7}, for any $n$, there are $c_n>0$ such that
\begin{equation}\label{4.19}
|\hat{x}_0 - \xi_0| \leq c_n\epsilon^n
\end{equation}
Combine \eqref{4.17}, \eqref{4.18} and \eqref{4.19}, we derive
\begin{equation}\label{partial sup}
\begin{aligned}
P^{\epsilon}\bigg( \sup_{t \leq \epsilon^{-b}} \sup_{|x-\xi_{0}| \leq 10^{-5}\zeta \epsilon^{-1}}|m_{i,t}(x)-\overline{m}_{\xi_{0}}(x)| < 2C\epsilon^{1/4};\\
 \sup_{|x-\xi_{0}| \leq 10^{-5}\zeta \epsilon^{-1}}|m_{i,\epsilon^{-b}}(x) - \overline{m}_{\xi_{0}}(x)| < 2\epsilon^{1/2-\hat{a}} \quad \forall i=1,2 \bigg)\geq 1- c_n\epsilon^n
\end{aligned}
\end{equation}
We now prove \eqref{4.18}. Let 
\begin{equation}
\hat{u}_t = \hat{m}_{1,t} - \overline{m}_{\hat{x}_{0}} \quad \hat{v}_t =\hat{m}_{2,t} - \overline{m}_{\hat{x}_{0}}
\end{equation}
Then, by Proposition \eqref{prop : 2D linearized sys} ,$u_t$ and $v_t$ evolve as the following system 
\begin{equation}\label{2D: new eqn 1}
\begin{aligned}
\hat{u}_t &= g_{t,\hat{x}_{0}}\hat{u}_0-\int_0^t  g_{t-s,\hat{x}_{0}}\Big( 3\overline{m}_{\hat{x}_{0}}\hat{u}_s^2 + \hat{u}_s^3 - \lambda(\hat{v}_s-\hat{u}_s) \Big)ds + \sqrt{\epsilon}W_{1,t,\hat{x}_{0}}\\
\hat{v}_t &= g_{t,\hat{x}_{0}}\hat{v}_0-\int_0^t  g_{t-s,\hat{x_0}} \Big( 3\overline{m}_{\hat{x}_{0}}\hat{v}_s^2 + \hat{v}_s^3 - \lambda(\hat{u}_s - \hat{v}_s) \Big)ds + \sqrt{\epsilon}W_{2,t,\hat{x}_{0}}
\end{aligned}
\end{equation}
 Given $a \leq \dfrac{1}{4}$, let $a',b > 0$ such that
\begin{equation} \label{condition a',b,a'+b/2}
  b<1/4, \; \quad  \hat{a}=a'+b/2<a, \quad b< \hat{a} 
\end{equation} 
Setting
\begin{equation}
B_{\epsilon} =  \lbrace \norm{W_{i,t,\hat{x}_0}}_\infty \leq \epsilon^{-a'}t^{1/2} \; \forall i=1,2 \rbrace  \cap_{i=1,2} \lbrace \sup_{t \leq \epsilon^{-b}}\norm{\hat{m}_{i,t}}_\infty \leq 2 \rbrace
\end{equation}
From \eqref{2 bound noise} and \eqref{bounded by 2},  $B_\epsilon$ has very large probability $(P^{\epsilon}(B_{\epsilon}) \geq 1 - c_n \epsilon^n)$.

We shall estimate $\hat{u}_t$ and $\hat{v}_t$ through $\hat{u}_t + \hat{v}_t$ and $\hat{u}_t - \hat{v}_t$. 
It is not difficult to get, in $B_{\epsilon}$, there exist $\alpha,C_1,C_2>0$
\begin{equation}
\norm{\hat{u}_t + \hat{v}_t}_\infty \leq  C_1 e^{-\alpha t}\norm{\hat{u}_0 + \hat{v}_0}_\infty+ C_2\int_0^t (\norm{\hat{u}_s}_\infty^2 + \norm{\hat{v}_s}_\infty^2)ds + 2\epsilon^{1/2-a'}t^{1/2} 
\end{equation}

In fact, by the sum of \eqref{2D: new eqn 1}
\begin{equation}
\begin{aligned}
\hat{u}_t + \hat{v}_t = g_{t,\hat{x}_{0}}(\hat{u}_0+\hat{v}_0)  -\int_0^t g_{t-s,\hat{x}_{0}}\Big( 3\overline{m}_{\hat{x}_{0}}(\hat{u}_s^2+\hat{v}_s^2) &+ \hat{u}_s^3 + \hat{v}_s^3 \Big)ds\\
&+ \sqrt{\epsilon}(W_{1,t,\hat{x}_{0}}+W_{2,t,\hat{x}_{0}})
\end{aligned}
\end{equation}
Since $\hat{x}_0$ is the center of $\hat{u}_0 + \hat{v}_0$ and by \eqref{asymptotic of operator g} and \eqref{1 asymptotic of operator g}, there are $\alpha, C_1,C_2>0$, in $B_\epsilon$
\begin{equation}\label{u_t + v_t}
\norm{\hat{u}_t + \hat{v}_t}_\infty \leq  C_1 e^{-\alpha t}\norm{\hat{u}_0 + \hat{v}_0}_\infty+ C_2\int_0^t (\norm{\hat{u}_s}_\infty^2 + \norm{\hat{v}_s}_\infty^2)ds + 2\epsilon^{1/2-a'}t^{1/2} 
\end{equation}
We now prove that, in $B_{\epsilon}$, there exists $C_4,C_5,C_6>0$
\begin{equation}\label{w_t = u_t - v_t} 
\norm{\hat{u}_t - \hat{v}_t}_\infty \leq C_4\epsilon^{-2\lambda t} \norm{\hat{u}_0 -\hat{v}_0}_\infty + C_5\epsilon^{1/2-a'}t^{1/2} + C_6\int_0^t (\norm{\hat{u}_s}_\infty^2 + \norm{\hat{u}_s}_\infty^2)ds
\end{equation}

In fact, by  substract of two equations in \eqref{2D: new eqn 1} 

\begin{equation}\label{2D: eqn u-v}
\begin{aligned}
\hat{u}_t - \hat{v}_t = g_{t,\hat{x}_0}(\hat{u}_0-\hat{v}_0) - &2\lambda \int_0^t g_{t-s,\hat{x}_0}(\hat{u}_s-\hat{v}_s)ds + \sqrt{\epsilon}(W_{1,t,\hat{x}_0}-W_{2,t,\hat{x}_0})\\
&+\int_0^t g_{t-s,\hat{x}_0}(3\overline{m}_{\hat{x}_0}(\hat{u}_s^2-\hat{v}_s^2) +(\hat{u}_s^3 -\hat{v}_s^3) )ds
\end{aligned}
\end{equation}
Denote $w_t:= \hat{u}_t-\hat{v}_t$ and define 
\begin{equation} \label{A_t}
A_t:=w_t+2\lambda \int_0^t g_{t-s,\hat{x}_{0}}w_s ds - g_{t,\hat{x}_0}w_0 
\end{equation}
thus, by \eqref{2D: eqn u-v},  
\begin{equation}
A_t = - \int_0^t g_{t-s,\hat{x}_{0}}(3\overline{m}_{\hat{x}_{0}}(\hat{u}_s^2-\hat{v}_s^2) + \hat{u}_s^3-\hat{v}_s^3)ds +\sqrt{\epsilon}(W_{1,t,\hat{x}_{0}}-W_{2,t,\hat{x}_{0}})
\end{equation}
and $A_t$ is bounded, in $B_\epsilon$, there is $C_3 >0$
\begin{equation}
 \norm{A_t}_{\infty} \leq   2\epsilon^{1/2-a'}t^{1/2} + C_3\int_0^t(\norm{\hat{u}_s}_\infty^2 + \norm{\hat{v}_s}_\infty^2)ds
\end{equation}
Moreover, by \eqref{A_t}, $w_t$ is represented through $A_t$,
\begin{equation}
\begin{aligned}
w_t &= e^{-2\lambda t}g_{t,\hat{x}_0}w_0 +A_t - 2\lambda \int_0^t e^{-2\lambda(t-s)}g_{t-s,\hat{x}_{0}}A_sds \\
\end{aligned}
\end{equation}
As a result, $w_t$ could be bounded by $A_t$,there exists $C_4, C_5 $ and $C_6 > 0$ 
\begin{equation}
 \norm{w_t}_\infty \leq C_4\epsilon^{-2\lambda t} \norm{w_0}_\infty + C_5\epsilon^{1/2-a'}t^{1/2} + C_6\int_0^t (\norm{\hat{u}_s}_\infty^{2} + \norm{\hat{v}_s}_\infty^2)ds
\end{equation}
Setting $z_t := \norm{\hat{u}_t}_\infty + \norm{\hat{v}_t}_\infty$, since \eqref{w_t = u_t - v_t}, \eqref{u_t + v_t} and $z_t \leq \norm{\hat{u}_t + \hat{v}_t}_\infty+ \norm{\hat{u}_t - \hat{v}_t}_\infty$, we obtain, there are $C_7,C_8 > 0$
\begin{equation}
\begin{aligned}\label{esti z_t}
z_t \leq C_1 e^{-\alpha t}\norm{\hat{u}_0 + \hat{v}_0}_\infty + C_4 e^{-2\lambda t}\norm{\hat{u}_0-\hat{v}_0}_\infty + C_7 \epsilon^{1/2-a'}t^{1/2} + C_8 \int_0^t z_s^2 ds
\end{aligned}
\end{equation}
or
\begin{equation}\label{estimation of z_t}
z_t \leq 2(C_1+C_4)(\norm{\hat{u}_0}_\infty + \norm{\hat{v}_0}_\infty) + C_7 \epsilon^{1/2-a'}t^{1/2} + C_8 \int_0^t z_s^2 ds
\end{equation}
The order of $\norm{\hat{u}_0}_\infty$ is $O(\epsilon^{1/4})$ as shown below
\begin{equation}\label{2D: order u_0}
\begin{aligned}
\norm{\hat{u}_0}_{\infty} &\leq \norm{\hat{m}_{1,0}-\overline{m}_{x_0}}_\infty + \norm{\overline{m}_{x_0} - \overline{m}_{\hat{x}_0}}_\infty \\
&\leq \norm{\hat{m}_{1,0}-\overline{m}_{\xi_0}}_\infty + \norm{\overline{m}_{\xi_0} - \overline{m}_{x_0}}_\infty + \norm{\overline{m}_{x_0} - \overline{m}_{\hat{x}_0}}_\infty  \\
&\leq \sup_{|x-\xi_0| \leq 10^{-4}\zeta \epsilon^{-1}} |\hat{m}_{1,0}-\overline{m}_{\xi_0}| + |\xi_0 - x_0| + |x_0 - \hat{x}_0| \\
&\leq \sup_{|x-\xi_0| \leq 10^{-4}\zeta \epsilon^{-1}} |m_{1,0}-\overline{m}_{\xi_0}| + 2c\epsilon^{1/4} + c_n \epsilon^n \\
&\leq \dfrac{C_9}{2(C_1 + C_4)}\epsilon^{1/4} \quad (\text{for some constant} \;C_9 \geq 1)
\end{aligned}
\end{equation}
In a similar way, there is $C_{10}>0$
\begin{equation}\label{2D: order v_0}
\norm{\hat{v}_0}_{\infty} \leq \dfrac{C_{10}}{2(C_1+C_4)}\epsilon^{1/4}
\end{equation}
Three estimations in \eqref{estimation of z_t},\eqref{2D: order u_0} and \eqref{2D: order v_0} implies
\begin{equation}\label{z_t}
z_t  \leq (C_9 + C_{10})\epsilon^{1/4} + C_7 \epsilon^{1/2-a'}t^{1/2} + C_8 \int_0^t z_s^2ds
\end{equation}
Thus, $\forall t \leq \epsilon^{-b}$, there is $C_{11}>0$
\begin{equation}
z_t \leq C_{11}\epsilon^{1/4} + C_8 \int_0^t z_s^2 ds, \quad C_{11}:= C_7 + C_9+C_{10}
\end{equation}
Setting $T:= \inf \lbrace t \geq 0: z_t \geq 2C_{11}\epsilon^{1/4} \rbrace$
. We prove $T \geq \epsilon^{-b}$ by contradiction. If $T < \epsilon^{-b}$ then
\begin{equation}
\begin{aligned}
C_{11}\epsilon^{1/4} &\leq 4C_8 C_{11}^2\int_0^T \epsilon^{1/2}ds\\
C_{11}\epsilon^{1/4} &\leq  4 C_8 C_{11}^2 \epsilon^{1/2-b}  
\end{aligned}
\end{equation}
which is impossible as $b < 1/4$ (by \eqref{condition a',b,a'+b/2}).
Replace $t=\epsilon^{-b}$ into \eqref{esti z_t} and by \eqref{condition a',b,a'+b/2}, we get there is $C_{12}>0$
\begin{equation}
\begin{aligned}
z_{\epsilon^{-b}} \leq &C_1 e^{-\alpha \epsilon^{-b}} \norm{\hat{u}_0 + \hat{v}_0}_\infty + C_4 \epsilon^{-2\lambda \epsilon^{-b}}\norm{\hat{u}_0 - \hat{u}_0}_\infty \\
&+ C_7\epsilon^{1/2-a'-b/2} + 4C_8C_{11}^2\epsilon^{1/2-b} \leq C_{12}\epsilon^{1/2-\hat{a}}
\end{aligned}
\end{equation}
The proof of \eqref{4.18} is complete. In $|x-\xi_0| > 10^{-5}\zeta \epsilon^{-1}$, by symmetry, it is sufficient to consider interval
\begin{equation}\label{interval L-,L+}
I=(\xi_0 + 10^{-6}\zeta \epsilon^{-1},(1 +10^{-6}\zeta)\epsilon^{-1})
\end{equation}
By \eqref{1}, we get $\forall i=1,2$
\begin{equation}
\begin{aligned}
\sup_{x \in I} |m_{i,0}(x) - 1| &\leq \sup_{x \in I}|m_{i,0}(x) - \overline{m}_{\xi_0}(x)| + \sup_{x \in I}|\overline{m}_{\xi_0}(x) - 1| \\
&\leq 2C\epsilon^{1/4}
\end{aligned}
\end{equation}
We define $\hat{m}_{i,0}$ as
\begin{equation}
\hat{m}_{i,0}(x) = \left\{\begin{array}{l}
m_{i,0}(x) \quad \text{if} \quad x \in I \\
m_{i,0}\bigg( (1+10^{-6}\zeta)\epsilon^{-1} \bigg) \quad \quad \text{if} \quad x\geq (1 +10^{-6}\zeta)\epsilon^{-1}  \\
m_{i,0}(\xi_0 + 10^{-6}\zeta \epsilon^{-1}) \quad \text{if} \quad x  < \xi_0 + 10^{-6}\zeta \epsilon^{-1}   
\end{array}\right.
\end{equation}
Then by Proposition \ref{proposition of barrier lemma}, 
\begin{equation}
P\bigg(\sup_{t \leq \epsilon^{-b}}\sup_{x_0 + 10^{-5}\zeta \epsilon^{-1} \leq x \leq \epsilon^{-1}} |m_{i,t}(x) - \hat{m}_{i,t}(x)| \leq c_n\epsilon^n\; \forall i=1,2 \bigg) \geq 1 - c_n\epsilon^n
\end{equation}
We now estimate $\hat{m}_{i,t}$, set $\tilde{u}_t = \hat{m}_{1,t} - 1$ and $\tilde{v}_t = \hat{m}_{2,t}-1$ then
\begin{equation}
\begin{aligned}
&(\partial_t - \dfrac{1}{2} \partial_{xx})(\tilde{u}_t-\epsilon^{1/2}Z_t) = -2\tilde{u}_t-6(\tilde{u}_t^2 + \tilde{u}_t^3) + \lambda(\tilde{v}_t-\tilde{u}_t)\\
&(\partial_t - \dfrac{1}{2} \partial_{xx})(\tilde{v}_t-\epsilon^{1/2}Z_t) = -2\tilde{v}_t-6(\tilde{v}_t^2 + \tilde{v}_t^3) + \lambda(\tilde{u}_t-\tilde{v}_t)
\end{aligned}
\end{equation}
Rewrite into corresponding integral system to $(\tilde{u}_t,\tilde{v}_t)$
\begin{equation}\label{u,v outside}
\begin{aligned}
\tilde{u}_t = e^{-2 t}H_t \tilde{u}_0 &- 6\int_0^t e^{-2\lambda (t-s)}H_{t-s}(\tilde{u}_s^2 + \tilde{u}_s^3)ds\\
 &+ \lambda \int_0^t e^{-2\lambda (t-s)}H_{t-s}(\tilde{v}_s-\tilde{u}_s)ds + \epsilon^{1/2}Z'_{1,t} \\
\end{aligned}
\end{equation}
\begin{equation}
\begin{aligned}
\tilde{v}_t = e^{-2 t}H_t \tilde{v}_0 &- 6\int_0^t e^{-2\lambda (t-s)}H_{t-s}(\tilde{v}_s^2 + \tilde{v}_s^3)ds \\
&+ \lambda \int_0^t e^{-2\lambda (t-s)}H_{t-s}(\tilde{u}_s-\tilde{v}_s)ds + \epsilon^{1/2}Z'_{2,t}
\end{aligned}
\end{equation}
where
\begin{equation}
Z_{i,t}' = Z_{i,t} - \int_0^t e^{-2(t-s)}H_{t-s}Z_{i,s}ds \quad \forall i=1,2
\end{equation}
We can evaluate $\tilde{u}_t,\tilde{v}_t$ similar to previous case and we then obtain the same estimate we proved before, we omit the details. So, \eqref{partial sup} is thus proved with the sup over $\Omega_\epsilon$. We have  
\begin{equation}
\begin{aligned}
&\sup_{t \leq \epsilon^{-b}}\norm{m_{i,t}-\overline{m}_{\xi_0}}_\epsilon \leq 2C\epsilon^{1/4}\quad \forall i=1,2 \\
\Rightarrow  &\sup_{t \leq \epsilon^{-b}}\norm{m_{i,t} - \overline{m}_{x_0}}\leq  \sup_{t \leq \epsilon^{-b}}\norm{m_{i,t}-\overline{m}_{\xi_0}}_\epsilon + \norm{\overline{m}_{\xi_0} - \overline{m}_{x_0}}_\epsilon \leq (2C+c)\epsilon^{1/4}
\end{aligned}
\end{equation}
and
\begin{equation}
\begin{aligned}
&\norm{m_{i,\epsilon^{-b}} - \overline{m}_{\xi_0}}_\epsilon \leq 2 \epsilon^{1/2-\hat{a}} \quad \forall i=1,2\\
\Rightarrow &\norm{\dfrac{m_{1,\epsilon^{-b}}+{m_{2,\epsilon^{-b}}}}{2} - \overline{m}_{\xi_0}}_\epsilon \leq 2 \epsilon^{1/2-\hat{a}}\\
\Rightarrow &|\xi_{\epsilon^{-b}} - \xi_0| \leq 2c\epsilon^{1/2-\hat{a}}
\end{aligned}
\end{equation}
As a result,
\begin{equation}
\begin{aligned}
\norm{m_{i,\epsilon^{-b}} - \overline{m}_{\xi_{\epsilon^{-b}}}}_\epsilon &\leq \norm{m_{i,\epsilon^{-b}}-\overline{m}_{\xi_0}}_\epsilon + \norm{\overline{m}_{\xi_0} - \overline{m}_{\xi_{\epsilon^{-b}}}}_\epsilon \\
 &\leq (2+2c)\epsilon^{1/2-\hat{a}} \leq \epsilon^{1/2-a} \quad \forall i=1,2
\end{aligned}
\end{equation}
\end{proof}

\subsection{Proof of Proposition \ref{weak convergence}}
\begin{proof}
We define the process $Y_t$  as a descretization of the process $X_{t}^{\epsilon}$
\begin{equation}
\begin{aligned}
Y_t &= \xi_{t_n} - x_0 \quad \text{where}\quad  t_n \leq t < t_{n+1}, t_n = nT_\epsilon \\
T_\epsilon &= n_\epsilon \epsilon^{-b}, \quad n_\epsilon = [\epsilon^{-1/10+b}], \quad \epsilon^{-1/10}-\epsilon^{-b} \leq T_\epsilon \leq \epsilon^{-1/10}
\end{aligned}
\end{equation}
and let $Y_\tau^\epsilon = Y_{\epsilon^{-1}\tau}$, then recall that $Y_\tau^\epsilon = X_\tau^\epsilon$ when $\tau = \epsilon t_n$. Let $F_t$ be the $\sigma$-algebra generated by the process $Z_{1,s},Z_{2,s}\; \forall s\leq t$.
To complete the proof, it is enough to show that $Y_\tau^\epsilon \rightarrow X_\tau$, which is a Brownian motion staring from $0$ and variance $3/8$,  weakly in $D[0,T]$ as $\epsilon \rightarrow 0$. Calling $\mathbb{P}^\epsilon$ the law on $D(\mathbb{R}^+,\mathbb{R})$ of $Y_\tau^\epsilon$. Following sufficient conditions was proved in Prop 3.7,\cite{BraDP}. 
\begin{lemma}
Given any $T>0$, the family of laws $\mathbb{P}^\epsilon$, $\epsilon>0$, on $D([0,T],\mathbb{R})$, is tight if there is $c>0$ so that for all $\epsilon$
\begin{equation}\label{k1}
\sup_{t_n \leq \epsilon^{-1}T} \mathbb{E}^{\epsilon}(\gamma_i(t_n)^2) \leq c \quad \forall i=1,2
\end{equation}
where
\begin{equation}
\begin{aligned}
\gamma_1(t_n) &= \epsilon^{-1}T_{\epsilon}^{-1}\mathbb{E}^{\epsilon}(Y_{t_{n+1}} - Y_{t_n}| F_{t_n})\\
\gamma_2 (t_n) &= \epsilon^{-1}T_{\epsilon}^{-1}\times\\
& \bigg( \mathbb{E}^{\epsilon}(Y_{t_{n+1}}^2 -Y_{t_n}^2| F_{t_n}) - [Y_{t_n}+\mathbb{E}^{\epsilon}(Y_{t_{n+1}}|F_{t_n}) ]\mathbb{E}^{\epsilon}(Y_{t_{n+1}}-Y_{t_n}|F_{t_n})\bigg)
\end{aligned}
\end{equation}
If \eqref{k1} holds and if 
\begin{equation}\label{k2}
\lim_{\epsilon \rightarrow 0} \sup_{t_n \leq \epsilon^{-1}T}\epsilon^{-1}T_{\epsilon}^{-1}E^{\epsilon}\bigg( |Y_{t_{n+1}}-Y_{t_n}|^4\bigg) = 0
\end{equation}
then any limit point $\mathbb{P}$ of $\mathbb{P}^\epsilon$ is supported by $C([0,T],\mathbb{R})$.
Finally, if \eqref{k1} and \eqref{k2} hold and if
\begin{equation}
\lim_{\epsilon \rightarrow 0} \sup_{t_n \leq \epsilon^{-1}T} E^\epsilon (\gamma_1(t_n)) = 0
\end{equation}
and
\begin{equation}
\lim_{\epsilon \rightarrow 0} \sup_{t_n \leq \epsilon^{-1}T} E^\epsilon \bigg( \bigg|\dfrac{3}{8}-\epsilon^{-1}T_{\epsilon}^{-1}E^\epsilon( Y_{t_{n+1}}^2 -Y_{t_n}^2 |F_{t_n})\bigg| \bigg) = 0
\end{equation}
then any limit point $\mathbb{P}$ is equal to $P$, the law of the Brownian motion with diffusion $D$ that starts from 0.
\end{lemma}
From \eqref{eqn 3.13} and \eqref{eqn 3.14}, we get
\begin{equation}\label{2D: noname}
\begin{aligned}
P^{\epsilon} \bigg( |\xi_{t_n}| \leq (1-\zeta/2)\epsilon^{-1}; \; \norm{m_{i,t_n} - \overline{m}_{\xi_{t_n}}}_{\epsilon} \leq \epsilon^{1/2-a} \; &\forall t_n \leq \epsilon^{-1}T, i=1,2 \bigg)\\
&\geq 1-c_n\epsilon^n
\end{aligned}
\end{equation}
Therefore, by using Markov property of $\xi_t$ and \eqref{2D: noname}, it suffices to check above criteria for process $(m_{1,t},m_{2,t})$ starting from $(m_{1,0},m_{2,0})$ and initial datum satisfies $\norm{m_{i,0}-\overline{m}_{x_0}}_\epsilon \leq \epsilon^{1/2-a}\; \forall i=1,2$ and $|x_0| \leq (1-\zeta/2)\epsilon^{-1}$. More explicitly, we need to prove
\begin{equation}\label{2d n1}
\mathbb{E}^{\epsilon} \bigg( \gamma_i(T_\epsilon)^2 \bigg) \leq c \quad \forall i=1,2
\end{equation} 
and
\begin{equation}\label{2d n2}
\lim_{\epsilon \rightarrow 0}\mathbb{E}^{\epsilon}\bigg( (\epsilon T_\epsilon)^{-1}(\xi_{T_\epsilon}-\xi_0)^4 + \gamma_1(T_\epsilon) + \bigg| \dfrac{3}{8}-(\epsilon T_\epsilon)^{-1}(\xi_{T_\epsilon}- \xi_0)^2 \bigg| \bigg) = 0
\end{equation}
The idea to confirm \eqref{2d n1} and \eqref{2d n2} is to work with process $(\overline{m}_{1,t},\overline{m}_{2,t})$ starting from $\overline{m}_{\xi_0}$ and $\xi_0 = \xi \bigg(\dfrac{m_{1,0} + m_{2,0}}{2} \bigg)$. Indeed, we are able to show that with "large probability", the distance $m_{i,T_{\epsilon}} - \overline{m}_{i,T_\epsilon}$ has order $O(\epsilon^{1-2a})$. As a result, by \eqref{first property of center}, with "large probability", the distance between two centers $\xi_{T_\epsilon} - \overline{\xi}_{T_\epsilon},\;\overline{\xi}_{T_\epsilon}:= \xi\bigg( \dfrac{\overline{m}_{1,T_\epsilon}+\overline{m}_{2,T_\epsilon}}{2}\bigg)$, has also order $O(\epsilon^{1-2a})$. Formal statement is given in Lemma \ref{lemma 1} below. Thus, we can prove \eqref{2d n1},\eqref{2d n2} in terms $"\overline{\gamma}_i(T_\epsilon),\overline{\xi}_{T_\epsilon}"$ instead of $"\gamma_i(T_\epsilon),\xi_{T_\epsilon}"$, where $\overline{\gamma}_i(T_\epsilon)$ is the variable $\gamma_i(T_\epsilon)$ in the new process starts from an instanton centered at $x_0$ with $|x_0| \leq (1-\zeta/2)\epsilon^{-1}$. 
Specifically, we have to prove
\begin{equation}\label{eqn 1}
 \mathbb{E}^\epsilon (\overline{\gamma}_1^{2}(T_\epsilon)) \leq c 
\end{equation}
\begin{equation}\label{eqn 2}
 \mathbb{E}^\epsilon (\overline{\gamma}_2^2(T_\epsilon)) \leq c \quad 
\end{equation}
\begin{equation}\label{eqn 3}
\lim_{\epsilon \rightarrow 0} (\epsilon T_\epsilon)^{-1} E^\epsilon((\overline{\xi}_{T_\epsilon} - x_0)^4) = 0
\end{equation}
\begin{equation}\label{eqn 4}
\lim_{\epsilon \rightarrow 0} E^\epsilon (\overline{\gamma}_1(T_\epsilon)) = 0
\end{equation}
and
\begin{equation}\label{eqn 5}
\lim_{\epsilon \rightarrow 0} E^\epsilon \bigg( |\dfrac{3}{8} - (\epsilon T_\epsilon)^{-1}|\overline{\xi}_{T_\epsilon}-x_0|^2| \bigg) = 0
\end{equation}
These conditions will hold if two following properties are claimed (proof given in Lemma \ref{lemma 2})
\begin{equation} \label{eqn 11}
|E(\overline{\xi}_{T_\epsilon}- x_0)| \leq c_n \epsilon^n
\end{equation}
and $\forall p > 0$, there is $C_p > 0$ such that
\begin{equation} \label{eqn 22}
E(|\overline{\xi}_{T_\epsilon}-x_0|^p) \leq C_p(\epsilon T_\epsilon)^{p/2}
\end{equation}
In fact, \eqref{eqn 1} and \eqref{eqn 4} follow from \eqref{eqn 11} since
\begin{equation}
\begin{aligned}
&\mathbb{E}^{\epsilon}(\overline{\gamma}_1(T_\epsilon))= (\epsilon T_\epsilon)^{-1} \mathbb{E}^{\epsilon}(\overline{\xi}_{T_\epsilon}-x_0) \\
&\mathbb{E}^{\epsilon}(\overline{\gamma}_1^2(T_\epsilon)) =  (\epsilon T_\epsilon)^{-2} \bigg( \mathbb{E}^\epsilon (\overline{\xi}_{T_\epsilon}-x_0) \bigg)^2 
\end{aligned}
\end{equation}
Next, \eqref{eqn 2} follows from \eqref{eqn 11} and let $p=2$ in \eqref{eqn 22}
\begin{equation}
\mathbb{E}^{\epsilon}(\overline{\gamma}_2^2(T_\epsilon)) = (\epsilon T_\epsilon)^{-2} \bigg( \mathbb{E}^\epsilon[(\overline{\xi}_{T_\epsilon}-x_0)^2] - (\mathbb{E}^\epsilon(\overline{\xi}_{T_\epsilon}-x_0) )^2 \bigg)^2
\end{equation}
On the other hand, \eqref{eqn 3} follows from \eqref{eqn 22} by choosing $p=4$. Finally, \eqref{eqn 5} follows from \eqref{2d ok} in Lemma \ref{lemma 2}. Therefore, it suffices to give proofs of Lemma \ref{lemma 1} and \ref{lemma 2}.
\begin{lemma}\label{lemma 1}\label{difference between two center}
For any $a,\zeta$ and $\epsilon$ positive, let $|x_0| \leq (1-\zeta/2)\epsilon^{-1}$ and assume $\norm{m_{i,0} - \overline{m}_{x_0}} \leq \epsilon^{1/2-a} \quad \forall i=1,2$ and suppose that $m_{1,0}$ and $m_{2,0}$ satisfy N.b.c in $\Omega_\epsilon$ and $\xi\bigg( \dfrac{m_{1,0}+m_{2,0}}{2} \bigg)$ $= x_0$. Let $m_{i,t}$ and $\overline{m}_{i,t}$ be the solutions of \eqref{eqn: integral_on_R_1}with initial datum respectively $m_{i,0}$ and $\overline{m}_{x_0}$. Then there is $c$ and for any $n, c_n$ so that 
\begin{equation}
P^{\epsilon}\bigg( \norm{m_{i,T_{\epsilon}} - \overline{m}_{i,T_\epsilon}}_\epsilon \leq c'\epsilon^{1-2a} \quad \forall i=1,2\bigg) \geq 1 - c_n\epsilon^n
\end{equation}
in addition, by \eqref{3.7} this follows that there exists $c_1 > 0$
\begin{equation}
P^{\epsilon} \bigg( |\xi_{T_\epsilon} - \overline{\xi}_{T_\epsilon}| \leq c_1\epsilon^{1-2a}\bigg) \geq 1 - c_n \epsilon^n
\end{equation}
\begin{proof}
Setting 
\begin{equation}
v_{i,t} := m_{i,t} -\overline{m}_{x_0}, \quad w_{i,t} := \overline{m}_{i,t} -\overline{m}_{x_0}, \quad u_{i,t} = v_{i,t} - w_{i,t} \quad \forall i =1,2
\end{equation} 
Then following \eqref{eqn: uu}, $v_1$  and $v_2$ satisfy the following system
\begin{equation}
\begin{aligned}
v_{1,t} &= g_{t,x_0}v_{1,0} - \int_0^t g_{t-s,x_0}\bigg( 3\overline{m}_{x_0}v_{1,s}^2  + v_{1,s}^3 + \lambda(v_{1,s}-v_{2,s})\bigg)ds + \sqrt{\epsilon}\hat{Z}_{1,t,x_0} \\
v_{2,t} &= g_{t,x_0}v_{2,0} - \int_0^t g_{t-s,x_0}\bigg( 3\overline{m}_{x_0}v_{2,s}^2  + v_{2,s}^3 + \lambda(v_{2,s}-v_{1,s})\bigg)ds + \sqrt{\epsilon}\hat{Z}_{2,t,x_0} \\
\end{aligned}
\end{equation}
Similarly, $w_1$ and $w_2$ satisfy the following system
\begin{equation}
\begin{aligned}
w_{1,t} &=  - \int_0^t g_{t-s,x_0}\bigg( 3\overline{m}_{x_0}w_{1,s}^2  + w_{1,s}^3 + \lambda(w_{1,s}-w_{2,s})\bigg)ds + \sqrt{\epsilon}\hat{Z}_{1,t,x_0} \\
w_{2,t} &=  - \int_0^t g_{t-s,x_0}\bigg( 3\overline{m}_{x_0}w_{2,s}^2  + w_{2,s}^3 + \lambda(w_{2,s}-w_{1,s})\bigg)ds + \sqrt{\epsilon}\hat{Z}_{2,t,x_0} \\
\end{aligned}
\end{equation}
It is easy to derive that $u_1$ and $u_2$ are solution of the following system
\begin{equation} \label{1D: u1, u2}
\begin{aligned}
u_{1,t} &= g_{t,x_1}u_{1,0} - \int_0^t g_{t-s,x_0}h_{1,s}u_{1,s}ds + \lambda\int_0^t g_{t-s,x_0}(u_{2,s}-u_{1,s})ds - D_{1,t}\\
u_{2,t} &= g_{t,x_0}u_{2,0} - \int_0^t g_{t-s,x_0}h_{2,s}u_{2,s}ds + \lambda\int_0^t g_{t-s,x_0}(u_{1,s}-u_{2,s})ds - D_{2,t}
\end{aligned}
\end{equation}
here
\begin{equation}
\begin{aligned}
D_{1,t} &:= \int_0^t g_{t-s,x_0}\bigg( 3u_{1,s}^2w_{1,s}+u_{1,s}^3 \bigg)ds, \quad  h_{1,t}:=-3\overline{m}_{x_0}(v_{1,t}+w_{1,t})-3w_{1,t}^2 \\
D_{2,t} &:= \int_0^t g_{t-s,x_0}\bigg( 3u_{2,s}^2w_{2,s}+u_{2,s}^3 \bigg)ds, \quad  h_{2,t}:=-3\overline{m}_{x_0}(v_{1,2}+w_{2,t})-3w_{2,t}^2
\end{aligned}
\end{equation}
Firstly, we prove that there is $c>0$
\begin{equation}\label{uniform bound m(i,t)-m_x0}
P^{\epsilon}\bigg( \norm{m_{i,t}- \overline{m}_{x_0}}_{\epsilon} \leq c\epsilon^{1/2-a}t^{1/2} \quad \forall t\leq \epsilon^{-1/10} \quad \forall i=1,2\bigg) \geq 1 -c_n\epsilon^n
\end{equation}
Consider $\hat{m}_{i,0}, \hat{m}_{i,t}$ and $\hat{u}_t, \hat{v}_t$ as in the proof of  Proposition \ref{prop 3.4}. 
Let $z_t := \norm{\hat{u}_t}_\infty + \norm{\hat{v}_t}_\infty$. Rewrite \eqref{estimation of z_t}
\begin{equation}
z_t \leq 2(C_1+C_4)(\norm{\hat{u}_0}_\infty + \norm{\hat{v}_0}_\infty) + C_7\epsilon^{1/2-a'}t^{1/2} + C_8 \int_0^t z_s^2 ds
\end{equation} 
Furthermore, by arguing as in \eqref{2D: order u_0}, we can prove that there are $c_1,c_2>0$
\begin{equation}
\norm{\hat{u}_0}_\infty \leq \dfrac{c_1}{2(C_1 + C_4)}\epsilon^{1/2-a}, \quad \norm{\hat{v}_0}_\infty \leq \dfrac{c_2}{2(C_1 + C_4)}\epsilon^{1/2-a}
\end{equation}
We have
\begin{equation}
z_t \leq c_3 \epsilon^{1/2-a} + C_7 \epsilon^{1/2-a'}t^{1/2} + C_8 \int_0^t z_s^2 ds, \quad c_3 := c_1 +  c_2
\end{equation}
and set 
\begin{equation}
\begin{aligned}
\gamma_t &:= 3\max{\bigg( c_3\epsilon^{1/2-a}, C_7\epsilon^{1/2-a'}t^{1/2}}\bigg) \\
T &:= \inf \lbrace{ t \geq 0: z_t \geq \gamma_t \rbrace}
\end{aligned}
\end{equation} 
We will prove that $T \geq \epsilon^{-1/10}$ by contradiction. If $T < \epsilon^{-1/10}$ then
\begin{equation}
\begin{aligned}
\gamma_T &\leq \dfrac{2}{3}\gamma_T  + C_8\int_0^T \gamma_s^2 ds\\
\gamma_T &\leq \dfrac{2}{3}\gamma_T  + C_8 \gamma_T^2 T\\
\Rightarrow 1/3 &\leq  C_8 \gamma_T T
\end{aligned}
\end{equation}
,which is impossible since
\begin{equation}
T\gamma_T \leq \epsilon^{1/2-a-3/20} < 1/3 \quad (\text{for }a<7/20)
\end{equation}
so $z_t \leq \gamma_t \leq 3C_7\epsilon^{1/2-a}t^{1/2} \quad \forall t \leq \epsilon^{-1/10}$. Recalling \eqref{4.18} and \eqref{partial sup}, we see that \eqref{uniform bound m(i,t)-m_x0} holds for the sup taken over $|x-x_0 | \leq 10^{-5}\zeta \epsilon^{-1}$. When $|x-x_0|>10^{-5}\zeta \epsilon^{-1}$, we do as in \eqref{interval L-,L+}-\eqref{u,v outside} and \eqref{uniform bound m(i,t)-m_x0} is proved. We omit the proof. Similarly, we also have the same bound for $\overline{m}_{i,t} - \overline{m}_{x_0}$ and we then obtain that there is $c>0$ such that
\begin{equation}\label{v,w,u}
\begin{aligned}
P^{\epsilon}\Big( \sum_{i=1,2}\norm{v_{i,t}}_\epsilon + \norm{w_{i,t}}_\epsilon + \norm{h_{i,t}}_\epsilon + \norm{u_{i,t}}_\epsilon \leq c\epsilon^{1/2-a}t^{1/2} \quad & \text{for all} \;t \leq \epsilon^{-1/10} \Big)\\
& \geq 1 - c_n\epsilon^n
\end{aligned}
\end{equation}
Let $U_t = u_{1,t} -u_{2,t}$ and we now prove that there is $\tilde{C}_3>0$ such that
\begin{equation}\label{bound u_1 - u_2}
\norm{U_t}_\epsilon \leq \tilde{C}_3 \epsilon^{1-2a}t^2 \quad \forall t \leq \epsilon^{-1/10}
\end{equation}
By \eqref{1D: u1, u2}, we have
\begin{equation}\label{U_t,A_t}
\begin{aligned}
U_t &= g_{t,x_0}U_0 -2\lambda \int_0^t g_{t-s,x_0}U_s ds + A_t
\end{aligned}
\end{equation}
where
\begin{equation}\label{bound A}
A_t = -\int_0^t g_{t-s,x_0}h_{1,s}u_{1,s}ds + \int_0^t g_{t-s,x_0}h_{2,s}u_{2,s}ds - D_{1,t}+D_{2,t}
\end{equation}
Since
\begin{equation} \label{2D: h_iu_i}
\bigg| \int_0^t g_{t-s,x_0}h_{i,s}u_{i,s}(x)ds \bigg|  \leq \kappa  t(c \epsilon^{1/2-a}t^{1/2})^2 \leq \tilde{C}_1 \epsilon^{1-2a}t^2 \quad \forall t \leq \epsilon^{-1/10}
\end{equation}
here $\kappa$ is in Lemma \ref{asymptotic of operator g}, and 
\begin{equation}
\begin{aligned} \label{2D: D_{i,t}}
|D_{i,t}(x)| &\leq 3 \bigg| \int_0^t g_{t-s,x_0}(u_{i,s}^2 w_{i,s})(x)ds \bigg| + \bigg| \int_0^t g_{t-s,x_0}u_{i,s}^3(x)ds \bigg| \\
&\leq 4t\kappa (\epsilon^{1/2-a}t^{1/2})^3 \leq \tilde{C}_2 \epsilon^{1-2a}t^2 \quad \forall t \leq \epsilon^{-1/10}
\end{aligned}
\end{equation}
by \eqref{v,w,u}, there is $\tilde{C}_3 > 0$
\begin{equation}\label{2D: A_t bd}
\norm{A_t}_\epsilon \leq \tilde{C}_3 \epsilon^{1-2a}t^2  \quad \forall t \leq \epsilon^{-1/10}
\end{equation}
$U_t$ is expressed in terms $A_t$, by \eqref{U_t,A_t},
\begin{equation}\label{bound U}
\begin{aligned}
&U_t = e^{-2\lambda t}g_{t,x_0} U_0 + A_t - 2\lambda \int_0^t e^{-2\lambda(t-s)}g_{t-s,x_0}A_s ds \\
\Rightarrow &\norm{U_t}_\epsilon \leq ce^{-2\lambda t}\norm{U_0}_\epsilon + C_4 \epsilon^{1-2a}t^2 \; (\eqref{2D: A_t bd})
\end{aligned}
\end{equation}
Thus, there is $\tilde{C}_5>0$
\begin{equation}
\norm{U_t}_\epsilon \leq \tilde{C}_5\epsilon^{1-2a}t^2 \quad \forall t \leq \epsilon^{-1/10}
\end{equation}
We now prove
\begin{equation}\label{bound u_1 + u_2}
\norm{u_{1,t} +u_{2,t}}_\epsilon \leq \tilde{C}_6 \epsilon^{1-2a}t^2 \quad \forall t \leq \epsilon^{-1/10}
\end{equation}
In fact, by \eqref{1D: u1, u2}
\begin{equation} \label{expression u_1,t + u_2,t}
\begin{aligned}
u_{1,t}+ u_{2,t} = g_{t,x_0}(u_{1,0} + u_{2,0})-\int_0^t g_{t-s,x_0}h_{1,s}u_{1,s}ds - &\int_0^t g_{t-s,x_0}h_{2,s}u_{2,s}ds\\
& - D_{1,t}-D_{2,t}
\end{aligned} 
\end{equation}
Since $x_0$ is the center of $u_{1,0} + u_{2,0}$ and \eqref{2D: h_iu_i},\eqref{2D: D_{i,t}}, there is $c^*>0$ such that
\begin{equation}
\norm{u_{1,t} + u_{2,t}}_\epsilon \leq c^*e^{-\alpha t}\norm{u_{1,0}+u_{2,0}}_\epsilon + (\tilde{C}_1+\tilde{C}_2) \epsilon^{1-2a}t^2 \leq \tilde{C}_6 \epsilon^{1-2a}t^2
\end{equation}
Since \eqref{bound u_1 - u_2} and \eqref{bound u_1 + u_2}, there is $\tilde{C}_7 > 0$ such that $\forall i=1,2$
\begin{equation}
\norm{u_{i,t}} \leq \tilde{C}_7 \epsilon^{1-2a}t^2 \quad \text{for all} \quad t \leq \epsilon^{-1/10}
\end{equation}
Since $\norm{h_{i,t}}_\infty \leq  c \epsilon^{1/2-a}\sqrt{t}$, it follows that
\begin{equation}\label{bound h*u}
\bigg| \int_0^t g_{t-s,x_0}h_{i,s}u_{i,s}(x)ds \bigg|  \leq \kappa t c\epsilon^{1/2-a}t^{1/2}\tilde{C}_7 \epsilon^{1-2a}t^2 \leq \tilde{C}_8 \epsilon^{1-2a} \quad \forall t \leq \epsilon^{-1/10}
\end{equation}
and
\begin{equation}\label{bound D_{i,t}}
\begin{aligned}
|D_{i,t}(x)| &\leq 3 \bigg| \int_0^t g_{t-s,x_0}(u_{i,s}^2 w_{i,s})(x)ds \bigg| + \bigg| \int_0^t g_{t-s,x_0}u_{i,s}^3(x)ds \bigg| \\
&\leq 3 \kappa t (C_7 \epsilon^{1-2a}t^2)^2 (c\epsilon^{1/2}t^{1/2}) + \kappa t (\tilde{C}_7 \epsilon^{1-2a}t^2)^3 \\
&\leq \tilde{C}_9 \epsilon^{1-2a} \quad \forall t\leq \epsilon^{-1/10}
\end{aligned} 
\end{equation}
Since \eqref{bound h*u}  and \eqref{bound A}, we get
\begin{equation}\label{2D A_t last}
\norm{A_t}_\epsilon \leq  \tilde{C}_{10} \epsilon^{1-2a} \quad \forall t \leq \epsilon^{-1/10}
\end{equation}
and from \eqref{bound U},\eqref{2D A_t last}
\begin{equation}\label{1d U_t exponentially bounded}
\norm{u_{1,t} - u_{2,t}}_\infty = \norm{U_t}_\infty \leq e^{-2\lambda t}\norm{u_{1,0}-u_{2,0}}_\infty + \tilde{C}_{11}\epsilon^{1-2a} \quad \forall t \leq \epsilon^{-1/10}
\end{equation}
Similarly, since \eqref{expression u_1,t + u_2,t} and \eqref{bound h*u} and \eqref{bound D_{i,t}}
\begin{equation}\label{1d u_t exponentially bounded}
\norm{u_{1,t} + u_{2,t}}_\infty \leq e^{-\alpha t}\norm{u_{1,0}+u_{2,0}}_\infty + \tilde{C}_{12} \epsilon^{1-2a} \quad \forall t \leq \epsilon^{-1/10}
\end{equation}
Finally, \eqref{1d U_t exponentially bounded} and \eqref{1d u_t exponentially bounded} implies the result. 
\end{proof}
\end{lemma}

\begin{lemma} \label{lemma 2}
Given any $\zeta > 0$, for any $\epsilon > 0$, $|x_0| \leq (1-\zeta/2)\epsilon^{-1}$, and let $(\overline{m}_{1,t},\overline{m}_{2,t})$ be the process that starts from $\overline{m}_{x_0}$. Setting
\begin{equation}
\overline{\xi}_{T_\epsilon} =  \xi\bigg(\dfrac{\overline{m}_{1,T_{\epsilon}} + \overline{m}_{2,T_{\epsilon}}}{2} \bigg) 
\end{equation}
Then for any $n$ there is $c_n$ so that
\begin{equation}
|E^{\epsilon}(\overline{\xi}_{T_\epsilon} - x_0)| \leq c_n \epsilon^n
\end{equation}
and for any positive $p$ there is $C_p$ so that
\begin{equation}\label{1D: estimate center p/2}
E^{\epsilon}(|\overline{\xi}_{T_\epsilon} - x_0|^p) \leq C_p(\epsilon T_\epsilon)^{p/2}
\end{equation}
\begin{proof}
We shall see that the symmetry of potential is crucial to the proof. Fix $\epsilon > 0$ and $x_0$, we define
\begin{equation}
\tilde{H}_t (x,y) = \left\{\begin{array}{l}
H_t(x,y) \quad \text{if} \quad |y-x_0| \leq 10^{-4}\zeta \epsilon^{-1} \\
0 \quad \quad \text{otherwise}\\  
\end{array}\right.
\end{equation}
and
\begin{equation}
\tilde{Z}_{i,t}(x) = \int_0^t \int_{\mathbb{R}}\tilde{H}_{t-s}(x,y)\alpha_i(y,s)dyds
\end{equation}
and let $(\tilde{m}_{1,t},\tilde{m}_{2,t})$ satisfies the following SPDE
\begin{equation}
\begin{aligned}
\tilde{m}_{1,t} &= \tilde{H}_t \overline{m}_{x_0} + \int_0^t \tilde{H}_{t-s}\bigg[ -V'(\tilde{m}_{1,s}) + \lambda(\tilde{m}_{2,s}-\tilde{m}_{1,s})\bigg]ds + \epsilon^{1/2}\tilde{Z}_{1,t} \\
\tilde{m}_{2,t} &= \tilde{H}_t \overline{m}_{x_0} + \int_0^t \tilde{H}_{t-s}\bigg[ -V'(\tilde{m}_{2,s}) + \lambda(\tilde{m}_{1,s}-\tilde{m}_{2,s})\bigg]ds + \epsilon^{1/2}\tilde{Z}_{2,t} 
\end{aligned}
\end{equation}
Let $G$ be the map from $C^0(\mathbb{R}) \rightarrow C^0(\mathbb{R})$
\begin{equation}
(Gf)(x) = -f(2x_0 - x)
\end{equation}
Then the process $\tilde{m}_{i,t}$ and $G\tilde{m}_{i,t}$ have the same distribution since the processes $\tilde{Z}_{i,t}$ and $G\tilde{Z}_{i,t}$ are equal in distribution. Thus, 
\begin{equation}
E^{\epsilon} (\tilde{\xi}_{T_\epsilon} -x_0) = 0, \quad  \tilde{\xi}_{T_\epsilon} := \xi \bigg( \dfrac{\tilde{m}_{1,T_{\epsilon}}+\tilde{m}_{2,T_{\epsilon}}}{2} \bigg)
\end{equation}

There is $\delta > 0$ and for any $n, c_n$ so that
\begin{equation}\label{2D distance noise}
P^{\epsilon} \bigg( \sup_{t \leq T_\epsilon} \sup_{|x-x_0| \leq 10^{-5}\zeta \epsilon^{-1}}|Z_{i,t}(x) - \tilde{Z}_{i,t}(x)| \leq e^{-\delta \epsilon^{-1}} \quad \forall i=1,2\bigg) \geq 1 - c_n \epsilon^n
\end{equation}
In the set lying in LHS of \eqref{2D distance noise}, there is $\delta_1 > 0$ 
\begin{equation}
\sup_{t \leq T_{\epsilon}} \sup_{|x-x_0| \leq 10^{-6}\zeta \epsilon^{-1}} |\overline{m}_{i,t}(x) -\tilde{m}_{i,t}(x)| \leq e^{-\delta_1 \epsilon^{-1}}\quad \forall i=1,2
\end{equation}
and by \eqref{3.7}, there is $\delta_2 > 0$
\begin{equation}
\begin{aligned}
|\overline{\xi}_{T_\epsilon} - \tilde{\xi}_{T_\epsilon}| & \leq \int_{\mathbb{R}}\overline{m}_{x_0}' (x)\bigg|\dfrac{\sum_{i=1,2} m_{i,T_\epsilon} (x)}{2}-  \dfrac{\sum_{i=1,2} \tilde{m}_{i,T_\epsilon}(x)}{2}\bigg|dx\\
& \leq \int_{|x-x_0| \leq 10^{-6}\zeta \epsilon^{-1}} \overline{m}_{x_0}'(x)e^{-\delta_1 \epsilon^{-1}}dx + 4\int_{|x-x_0| \geq 10^{-6}\zeta \epsilon^{-1}}\overline{m}_{x_0}'(x)dx \\
&\leq e^{-\delta_2 \epsilon^{-1}}
\end{aligned}
\end{equation}
Thus, for any $n$ there is $c_n$
\begin{equation}
|E^{\epsilon}(\overline{\xi}_{T_\epsilon}-x_0)|  = |E(\overline{\xi}_{T_\epsilon} - \tilde{\xi}_{T_\epsilon})|\leq c_n \epsilon^n
\end{equation}
 Setting 
\begin{equation}
v_t = \overline{m}_{2,t} - \overline{m}_{x_0},\quad u_t = \overline{m}_{1,t} - \overline{m}_{x_0}, u_0 = v_0 = 0
\end{equation}
then $(u_t,v_t)$ satisfies \eqref{prop : 2D linearized sys} with $u_0 =0 , v_0 = 0$ and suppose that $W_{i,t,x_0} \in G_\epsilon(a,x_0)$, i.e, $\norm{W_{i,t,x_0}} \leq \epsilon^{1/2-a}t^{1/2}$. We get, by \eqref{eqn: uu} and \eqref{uniform bound m(i,t)-m_x0}
\begin{equation}
 \norm{u_t + v_t - \Big( W_{1,t,x_0} +W_{2,t,x_0}\Big)}_\infty  \leq c_1 \epsilon^{1-2a}T_\epsilon^2 \quad \forall t \leq T_\epsilon
\end{equation}
and from Prop \ref{prop noise}, we have $W_{i,t,x_0} = B_{i,t,x_0}\tilde{m}_{x_0}' + R_{i,t,x_0} = \sqrt{D} B_{i,t}\overline{m}_{x_0}' + R_{i,t,x_0}, \; D= 3/4,\quad B_{t,x_0} : = B_{1,t,x_0}+ B_{2,t,x_0} $ and it follows that
\begin{equation}\label{center estimate u}
\norm{u_t + v_t - \sqrt{D}\epsilon^{1/2}B_{t,x_0}\overline{m}_{x_0}'}_\infty \leq  c_1 \epsilon^{1-2a}T_\epsilon^2 + \epsilon^{1/2} \norm{R_{1,T_\epsilon,x_0}}_\infty + \epsilon^{1/2} \norm{R_{2,T_\epsilon,x_0}}_\infty 
\end{equation}
Furthermore, by using linear approximation 
\begin{equation} \label{center estimate m bar}
\begin{aligned}
\norm{\overline{m}_{x_0+ \sqrt{D}/2\epsilon^{1/2} B_{T_\epsilon,x_0}} -  \overline{m}_{x_0} - \dfrac{\sqrt{D}}{2}\epsilon^{1/2} B_{T_\epsilon,x_0} \overline{m}_{x_0}'}_\infty \leq  c_2\epsilon B_{T_\epsilon,x_0}^2  
\end{aligned}
\end{equation}
From \eqref{center estimate m bar} and \eqref{center estimate u}, we get
\begin{equation}
\norm{\dfrac{\overline{m}_{1,T_\epsilon} + \overline{m}_{2,T_\epsilon}}{2}- \overline{m}_{x_0 + \sqrt{D}/2\epsilon^{1/2}B_{T_\epsilon,x_0}}}_\infty  \leq c_1 \epsilon^{1-2a}T_\epsilon^2 + 2\epsilon^{1/2-a} + c_2 \epsilon B_{T_\epsilon,x_0}^2
\end{equation}
and by \eqref{first property of center},
\begin{equation}
 \bigg|\overline{\xi}_{T_\epsilon}  - (x_0 +\dfrac{\sqrt{D}}{2}\epsilon^{1/2}B_{T_\epsilon,x_0})\bigg| \leq   c \Bigg( \epsilon^{1-2a}T_\epsilon^2 + 2 \epsilon^{1/2-a} + \epsilon B_{T_\epsilon,x_0}^2 \Bigg)
\end{equation} 
or
\begin{equation}\label{2d ok}
\begin{aligned}
\bigg| \epsilon^{-1/2}T_{\epsilon}^{-1/2}&\bigg( \overline{\xi}_{T_\epsilon} -x_0 \bigg) - \sqrt{\dfrac{D}{2}}\dfrac{B_{T_\epsilon,x_0}}{\sqrt{2} T_{\epsilon}^{1/2}}\bigg|\\
&\leq   c \Bigg( \epsilon^{1/2-2a}T_\epsilon + 2 T_\epsilon^{-1/2}\epsilon^{-a} + \epsilon^{1/2}T_{\epsilon}^{-1/2} B_{T_\epsilon,x_0}^2 \Bigg)
\end{aligned}
\end{equation}
Since the RHS $\rightarrow 0$ as $\epsilon \rightarrow 0$ and $\dfrac{1}{\sqrt{2}} \dfrac{B_{T_\epsilon,x_0}}{T_\epsilon^{1/2}} \sim  N(0,\dfrac{D_{\epsilon} + D_{\epsilon}}{2}) \sim N(0,D_\epsilon) \approx N(0,1)$, we derive \eqref{1D: estimate center p/2}.
\end{proof}
\end{lemma}
\end{proof}

\section{Appendix}
\subsection{Comparision theorem}
\begin{theorem}\label{2D:Comparision theorem} (Comparision theorem).

Let $(m_{1,t},m_{2,t})$ and $(\overline{m}_{1,t},\overline{m}_{2,t})$ solve  \eqref{eqn: integral_on_R_1} with initial data $(m_{1,0},m_{2,0})$ and $(\overline{m}_{1,0},\overline{m}_{2,0})$ and with the same noise $Z_{1,t}$ and $Z_{2,t}$ respectively. If $m_{i,0} \geq \overline{m}_{i,0} \; \forall i=1,2$, then $m_{i,t} \geq \overline{m}_{i,t} \quad \forall t\geq 0 \quad \forall i=1,2$.
\begin{proof} By using basic expansions, we derive for any $a>0$
\begin{equation}
\begin{aligned}
m_{1,t} = e^{-at}H_t m_{1,0} &+ \int_0^t e^{-a(t-s)}H_{t-s}\bigg( -V'(m_{1,s})+(a-\lambda)m_{1,s} + \lambda m_{2,s} \bigg)ds \\
&+ \epsilon^{1/2}\bigg(Z_{1,t} - a \int_0^t e^{-a(t-s)}H_{t-s}Z_{1,s}ds\bigg)
\end{aligned}
\end{equation}
\begin{equation}
\begin{aligned}
m_{2,t} = e^{-at}H_t m_{2,0} &+ \int_0^t e^{-a(t-s)}H_{t-s}\bigg( -V'(m_{2,s})+(a-\lambda)m_{2,s} + \lambda m_{1,s} \bigg)ds \\
&+ \epsilon^{1/2}\bigg(Z_{2,t} - a \int_0^t e^{-a(t-s)}H_{t-s}Z_{2,s}ds\bigg)
\end{aligned}
\end{equation}
Setting $w_{i,t} = m_{i,t} - \overline{m}_{i,t} \quad \forall i=1,2$,
\begin{equation}\label{2D: iterating method}
\begin{aligned}
&w_{1,t} = H_t w_{1,0} + \int_0^t H_{t-s} \bigg( (F_{1,s}+ a- \lambda)w_{1,s} + \lambda w_{2,s} \bigg)ds \\
&w_{2,t} = H_t w_{2,0} + \int_0^t H_{t-s} \bigg( (F_{2,s}+ a- \lambda)w_{2,s} + \lambda w_{1,s} \bigg)ds \\
\end{aligned}
\end{equation}
where 
\begin{equation}
F_{i,s} =  - \dfrac{V'(m_{i,s}) - V'(\overline{m}_{i,s})}{m_{i,s}-\overline{m}_{i,s}} \quad \forall i=1,2 
\end{equation}
Since $m_{i,t}, \overline{m}_{i,t}$ are bounded $\forall t>0$, this also implies that $F_{1,s}$ and $F_{2,s}$ are bounded, so there exists $a>0$ such that $F_{i,s} + a -\lambda > 0 \quad \forall s \leq T \quad \forall i=1,2$. By applying iterating method in \eqref{2D: iterating method} with $F_{1,s}$ and $F_{2,s}$ are known, we get a series of non negative terms, which implies the result.  
\end{proof}
\end{theorem}

\begin{lemma}
Suppose that $\norm{m_{1,0}}_{\infty}\;$and$\; \norm{m_{2,0}}_{\infty} \leq 1+\dfrac{1}{32(2+\lambda)}$. Then there are $c'$ and $c >0$ so that if $m_1$ and $m_2$ solves $\eqref{eqn: integral_on_R_1}$, then
\begin{equation}\label{bounded by 2}
\begin{aligned}
\mathbb{P}^{\epsilon}\bigg(\sup_{t \leq \epsilon^{-2}} \norm{m_{1,t}}_{\infty} > 2\bigg) \leq c'e^{-c\epsilon^{-1}}  \\
\mathbb{P}^{\epsilon}\bigg(\sup_{t \leq \epsilon^{-2}} \norm{m_{2,t}}_{\infty} > 2\bigg) \leq c'e^{-c\epsilon^{-1}} 
\end{aligned}
\end{equation}
\begin{proof}
By Theorem \ref{2D:Comparision theorem}, it is sufficient to solve the problem for the special case $m_1(x,0)=m_2(x,0)=1+\dfrac{1}{32(2+\lambda)}$. 
Replacing $\overline{u}_t = m_{1,t} - 1$ and $\overline{v}_t=m_{2,t} - 1$ into \eqref{classical system}, then
\begin{equation}
\partial_t \overline{u}_t -\dfrac{1}{2}\partial_{xx}\overline{u}_t + (2+\lambda)\overline{u}_t = -3\overline{u}_t^{2} -\overline{u}_t^3 + \lambda \overline{v}_t + \sqrt{\epsilon} \dot{\alpha}_{1,t}
\end{equation}
\begin{equation}
\partial_t \overline{v}_t -\dfrac{1}{2}\partial_{xx}\overline{v}_t + (2+\lambda)\overline{v}_t = -3\overline{v}_t^2 -\overline{v}_t^3 + \lambda \overline{u}_t + \sqrt{\epsilon} \dot{\alpha}_{2,t}
\end{equation}
\begin{equation}
\overline{u}(x,0) = \overline{v}(x,0) = \dfrac{1}{32(2+\lambda)}
\end{equation}
It is seen that $(\partial_t  - \dfrac{1}{2}\partial_{xx} + (2+ \lambda)Id)^{-1} = e^{-(2+\lambda)t}H_t^{(\epsilon)}$ and the corresponding integral equation for $u_t$ and $v_t$ is 
\begin{equation}\label{1D: compare u,v}
\begin{aligned}
&\overline{u}_t = e^{-(2+\lambda)t}H_t^{(\epsilon)} \overline{u}_0 + \int_0^t e^{-(2+\lambda)(t-s)}H_{t-s}\Big(-3\overline{u}_s^2 - \overline{u}_s^3 + \lambda \overline{v}_s\Big)ds + \sqrt{\epsilon}W_t \\
&\overline{v}_t = e^{-(2+\lambda)t}H_t^{(\epsilon)} \overline{v}_0 + \int_0^t{e^{-(2+\lambda)(t-s)}}H_{t-s}\Big(-3\overline{v}_s^2 - \overline{v}_s^3 + \lambda \overline{u}_s\Big)]ds + \sqrt{\epsilon}Q_t
\end{aligned}
\end{equation}
where
\begin{equation}
W_t(x) = \int_0^t \int_{-\epsilon^{-1}}^{\epsilon^{-1}}  e^{-(2+\lambda)(t-s)}H_{t-s}^{(\epsilon)}(x,y)\dot{\alpha_1}(y,s)dyds
\end{equation}
\begin{equation}
Q_t(x) = \int_0^t \int_{-\epsilon^{-1}}^{\epsilon^{-1}}  e^{-(2+\lambda)(t-s)}H_{t-s}^{(\epsilon)}(x,y)\dot{\alpha_2}(y,s)dyds
\end{equation}
here $H_t^{(\epsilon)}$ is the heat operator with Neumann boundary condition on $\Omega_\epsilon$. The boundedness of Gaussian processes $W_t$ and $Q_t$ is investigated in \cite{Walsh1981}, and there are positive constants $k_1$ and $k_2$ such that 
\begin{equation}
 \mathbb{P}^{\epsilon} \bigg(\sup_{t \leq \epsilon^{-2}, x \in \mathbb{R}}|\sqrt{\epsilon}W_t(x)| > \dfrac{1}{32(2+\lambda)} \bigg) \leq k_1 e^{-k_2 \epsilon^{-1}}
\end{equation}
\begin{equation}
\mathbb{P}^{\epsilon} \bigg(\sup_{t \leq \epsilon^{-2}, x \in \mathbb{R}}|\sqrt{\epsilon}Q_t(x)| > \dfrac{1}{32(2+\lambda)}\bigg) \leq k_1 e^{-k_2 \epsilon^{-1}}
\end{equation}
Denoting by
\begin{equation}
T_{\delta} = \inf{ \lbrace t \geq 0 : \sup_{0\leq s \leq t}\norm{\overline{u}_s}_{\infty} + \sup_{0 \leq s \leq t}\norm{\overline{v}_s}_{\infty} \geq \delta \rbrace}
\end{equation} 
\begin{equation}
\dfrac{1}{16(2+\lambda)}=\overline{u}_0 + \overline{v}_0 < \delta < 1
\end{equation}
Then
\begin{equation}
\begin{split}
&\mathbb{P}^{\epsilon}\bigg(\sup_{t \leq \epsilon^{-2}} \norm{m_{1,t}}_{\infty} > 2 \bigg) \\
&\leq \mathbb{P}^{\epsilon}\bigg(\sup_{t \leq \epsilon^{-2}}  \norm{\overline{u}_t}_{\infty} > 1 \bigg) \\
&\leq \mathbb{P}^{\epsilon} \bigg(T_{\delta} \leq \epsilon^{-2}\bigg)\\
&\leq \mathbb{P}^{\epsilon} \bigg(T_{\delta} \leq \epsilon^{-2},\sup_{t \leq \epsilon^{-2}}{\norm{\sqrt{\epsilon}W_t}_{\infty}} \leq \dfrac{1}{32(2+\lambda)}, \sup_{t \leq \epsilon^{-2}}{\norm{\sqrt{\epsilon}Q_t}_{\infty}} \leq \dfrac{1}{32(2+\lambda)} \bigg) \\
&+ \mathbb{P}^{\epsilon} \bigg(T_{\delta} \leq \epsilon^{-2},\sup_{t \leq \epsilon^{-2}}{\norm{\sqrt{\epsilon}W_t}_{\infty}} > \dfrac{1}{32(2+\lambda)} \cup \sup_{t \leq \epsilon^{-2}}{\norm{\sqrt{\epsilon}Q_t}_{\infty}} > \dfrac{1}{32(2+\lambda)} \bigg) \\
&=: \mathbb{P}^{\epsilon}(A) + \mathbb{P}^{\epsilon}(B).
\end{split}
\end{equation}
$\mathbb{P}^{\epsilon}(B)$ is estimated as follows
\begin{equation}
\begin{split}
\mathbb{P}^{\epsilon}(B) &\leq \mathbb{P}^{\epsilon} \bigg(\sup_{t \leq \epsilon^{-2}}{\norm{\sqrt{\epsilon}W_t}_{\infty}} > \dfrac{1}{32(2+\lambda)} \bigg)+ \mathbb{P}^{\epsilon} \bigg(\sup_{t \leq \epsilon^{-2}}{\norm{\sqrt{\epsilon}Q_t}_{\infty}} > \dfrac{1}{32(2+\lambda)}\bigg) \\
& \leq 2k_1e^{-k_2 \epsilon^{-1}}
\end{split}
\end{equation}
We now claim that $\mathbb{P}^{\epsilon}(A)= 0$ to complete the proof.

If $T \leq \epsilon^{-2}, \; \sup_{t \leq \epsilon^{-2}}{\norm{\sqrt{\epsilon}W_t}_{\infty}} \leq \dfrac{1}{32(2+\lambda)} \;\text{and} \;\sup_{t \leq \epsilon^{-2}}{\norm{\sqrt{\epsilon}Q_t}_{\infty}} \leq \dfrac{1}{32(2+\lambda)}$, from \eqref{1D: compare u,v}, we get
\begin{equation}
\begin{aligned}
&\sup_{0 \leq s \leq T_{\delta}}\norm{\overline{u}_s}_{\infty}\\
 &\leq \dfrac{1}{32(2+\lambda)} + \Big(3\delta^2 + \delta^3 + \lambda \sup_{0 \leq s \leq T_{\delta}} \norm{\overline{v}_s}_\infty\Big) \times \\
 &\times \underbrace{\sup_{x \in \mathbb{R}} \int_0^{T_{\delta}} dse^{-(2+\lambda)(T_{\delta}-s)}\int_{-\infty}^{\infty}dy H_{T_{\delta}-s}(x,y)}_{\leq 1/(2+\lambda)} +\dfrac{1}{32(2+\lambda)} \\
&\leq \dfrac{1}{32(2+\lambda)} + \dfrac{1}{2+\lambda}(3\delta^2 + \delta^3)  + \dfrac{\lambda}{2+\lambda} \sup_{0 \leq s \leq T_{\delta}}\norm{\overline{v}_s}+ \dfrac{1}{32(2+\lambda)}
\end{aligned}
\end{equation}
Similarly,
\begin{equation}
  \sup_{0 \leq s \leq T_{\delta}}\norm{\overline{v}_s}_{\infty} \leq \dfrac{1}{32(2+\lambda)} + \dfrac{1}{2+\lambda}(3\delta^2 + \delta^3)  + \dfrac{\lambda}{2 + \lambda}\sup_{0 \leq s \leq T_{\delta}}\norm{\overline{u}_s}_\infty + \dfrac{1}{32(2+\lambda)}
\end{equation}
Therefore,
\begin{equation}
\begin{aligned}
\sup_{0 \leq s \leq T_{\delta}}\norm{\overline{u}_s}_{\infty}  + \sup_{0 \leq s \leq T_{\delta}}\norm{\overline{v}_s}_{\infty}  = \delta &\leq \dfrac{1}{16(2+\lambda)}+\dfrac{2}{2+\lambda}(3\delta^2 + \delta^3)\\
&+ \dfrac{\lambda}{\lambda+2} \delta + \dfrac{1}{16(2+\lambda)}
\end{aligned}
\end{equation}
Thus
\begin{equation}
\dfrac{2}{\lambda+2} \delta \leq \dfrac{1}{16(2+\lambda)}+ \dfrac{2}{2+\lambda}(3\delta^2 + \delta^3) + \dfrac{1}{16(2+\lambda)}
\end{equation}
which is impossible if  $\delta=\dfrac{1}{8}$.
\end{proof}
\end{lemma}

\subsection{Barrier lemma}
\begin{prop}\label{proposition of barrier lemma}(The barrier Lemma.) There are $V>0$ and c ($V$ and $c$ does not depend on $\lambda$)so that the following holds. Let $(m_{1,t},m_{2,t})$ and $(m_{1,t}^*,m_{2,t}^*)$ both solve \eqref{eqn: integral_on_R_1}, with initial conditions respectively $(m_{1,0},m_{2,0})$ and $(m_{1,0}^*,m_{2,0}^*)$. Suppose that for some $T>0$ their sup norms for $t\leq T$ are bounded by 2 and that $m_{i,0}(x) = m_{i,0}^*(x)$ for all $|x| \leq VT.$ Then for $i=1,2$
\begin{equation} \label{barrier lemma}
\sup_{t \leq T}|m_{i,t}(0) - m_{i,t}^*(0)| \leq ce^{-T}
\end{equation}
\begin{proof} $m_{1,t} + m_{2,t}$ and $m_{1,t}^* + m_{2,t}^*$ satisfy
\begin{equation} \label{2D: barrier lemma SUM}
\begin{aligned}
&m_{1,t}+m_{2,t} = H_t (m_{1,0}+m_{2,0}) - \int_0^t H_{t-s}(V'(m_{1,s})+V'(m_{2,s}))ds \\
&m_{1,t}^* + m_{2,t}^* = H_t (m_{1,0}^*+m_{2,0}^*) - \int_0^t H_{t-s}(V'(m_{1,s}^*)+V'(m_{2,s}^*))ds
\end{aligned} 
\end{equation}
Setting 
\begin{equation}
\begin{aligned}
&p_{t} := m_{1,t}+m_{2,t}-m_{1,t}^*-m_{2,t}^* \\
&q_{t} := m_{1,t}-m_{2,t}-m_{1,t}^*+m_{2,t}^*
\end{aligned}
\end{equation}
By \eqref{2D: barrier lemma SUM} and $\norm{m_{i,t}}_\infty \leq 2$, there is $c_1>0$ such that
\begin{equation}
|p_t| \leq H_t |p_0| + c_1 \int_0^t H_{t-s}\bigg( |m_{1,s}-m_{1,s}^*|+|m_{2,s}-m_{2,s}^*| \bigg)ds  \quad \forall t \leq T 
\end{equation}
$m_{1,t} - m_{2,t}$ and $m_{1,t}^* - m_{2,t}^*$ evolve as
\begin{equation}\label{2D: barrier lemma SUBTRACT}
\begin{aligned}
m_{1,t}-m_{2,t} &= (e^{-2\lambda t}H_t) (m_{1,0}-m_{2,0})\\
&- \int_0^t e^{-2 \lambda (t-s)}H_{t-s}(V'(m_{1,s})-V'(m_{2,s}))ds 
\end{aligned} 
\end{equation}
and
\begin{equation}
\begin{aligned}
m_{1,t}^* - m_{2,t}^* &=  (e^{-2\lambda t}H_t) (m_{1,0}^*-m_{2,0}^*)\\
&- \int_0^t e^{-2 \lambda (t-s)}H_{t-s}(V'(m_{1,s}^*)-V'(m_{2,s}^*))ds
\end{aligned}
\end{equation}
There is $c_2>0$ such that
\begin{equation}
\begin{aligned}
|q_t(x)| \leq &(e^{-2\lambda t}H_t) (|q_0|)(x)\\
&+ c_2\int_0^t e^{-2\lambda (t-s)}H_{t-s}\bigg( |m_{1,s}-m_{1,s}^*|+|m_{2,s}-m_{2,s}^*| \bigg) (x)ds 
\end{aligned}
\end{equation}
Since $e^{-2 \lambda t}$ and $e^{-2\lambda (t-s)} \; (s<t)$ are bounded by $1$, we get 
\begin{equation}
|q_t(x)| \leq (H_t |q_0|)(x) + c_2\int_0^t H_{t-s}\bigg( |m_{1,s}-m_{1,s}^*|+|m_{2,s}-m_{2,s}^*| \bigg) (x)ds 
\end{equation}
Since $m_{1,t}- m_{1,t}^* = \dfrac{1}{2}(p_t + q_t)$ and $m_{2,t}-m_{2,t}^*= \dfrac{1}{2}(p_t - q_t)$, it is easy to derive 
\begin{equation}\label{1D: p_t + q_t}
|m_{1,t} - m_{1,t}^*| + |m_{2,t}-m_{2,t}^*| \leq |p_t| + |q_t|
\end{equation}
As a result, let $w_t = |p_t| + |q_t|  $, there is $c_3 > 0$ such that
\begin{equation} \label{w_t}
w_t(x) \leq (H_tw_0)(x) + c_3\int_0^t ds (H_{t-s}w_s)(x)ds
\end{equation}
Iterating $M$ times \eqref{w_t},
\begin{equation}\label{1D: iterating M times}
\begin{aligned}
w_t(0) &\leq (H_tw_0)(0) + \sum_{n=1}^M \dfrac{(c_3t)^n}{n!}(H_tw_0)(0) + 2 \dfrac{(c_3t)^M}{M!}\\
&\leq e^{c_3 t}(H_t w_0)(0) + 2 \dfrac{(c_3 T)^M}{M!}\quad \forall t\leq T
\end{aligned} 
\end{equation}
By the Stirling formula
\begin{equation}
n! = \sqrt{2 \pi}e^{-n}n^{n+1/2}\bigg( 1+ 0(\dfrac{1}{\sqrt{n}})\bigg)
\end{equation}
The last term in \eqref{1D: iterating M times} is bounded as follows, choosing $M \geq e^2 c_3 t$, then
\begin{equation}\label{2D ss 1}
\begin{aligned}
\dfrac{(c_3 T)^M}{M!} &\leq C \exp \bigg(M \log (c_3 T) -M(\log{M}-1) \bigg) \leq C \exp \bigg({-M \log{\dfrac{M}{ec_3 T}}} \bigg) \\
&\leq C e^{ -e^2c_3 T }
\end{aligned}
\end{equation}
Furthermore,
\begin{equation}\label{2D ss 2}
\begin{aligned}
(H_tw_0)(0) &= \int_{\mathbb{R}} H_t(0,y)w_0(y)dy = \int_{|y| \geq VT}H_t(0,y)w_0(y)dy \\
&\leq c_4 \int_{|y| \geq VT}H_t(0,y)dy \leq 2c_4 \dfrac{1}{V\sqrt{T}}e^{-V^2 T} \quad \forall t \leq T
\end{aligned}
\end{equation}
for some $c_4 > 0$. This combines with \eqref{1D: iterating M times} and \eqref{2D ss 1}, we get
\begin{equation}
w_t(0) \leq 2c_4 \dfrac{1}{V\sqrt{T}} e^{-T(V^2-c_3)} + 2C e^{-e^2 c_3 T}
\end{equation}
Finally, by choosing $V^2 \geq c_3 +1$ and note that \eqref{1D: p_t + q_t}, we obtain \eqref{barrier lemma}.
\end{proof}
\end{prop}
\begin{corro}
If 
\begin{equation}
m_{i,0}(x) = m_{i,0}^*(x) \quad \forall\; |x|\leq L+VT \quad \forall i=1,2
\end{equation} 
then
\begin{equation}
\sup_{t \leq T}\sup_{|x|\leq L}|m_{i,t}(x) - m_{i,t}^*(x)| \leq c e^{-T} \quad \forall i=1,2
\end{equation}
\end{corro}
\section*{Acknowledgement}
The author is grateful to Prof. Errico Presutti for having suggested the problem and useful discussions.
\phantomsection
\addcontentsline{toc}{chapter}{Bibliography}\bibliographystyle{plain}
\bibliography{Thesis_ref}
\Addresses
\end{document}